\documentclass[11pt, reqno]{amsart}
\usepackage{amssymb,latexsym,amsmath,amsfonts}
\usepackage{supertabular}
\usepackage{enumerate}
\usepackage[mathscr]{eucal}
\usepackage{mathrsfs}
\usepackage{enumitem}
\usepackage{hyperref}
\usepackage{mathtools}
\usepackage{amsmath}

\usepackage[nameinlink]{cleveref}

\allowdisplaybreaks

\numberwithin{equation}{section}
\theoremstyle{definition}
\newtheorem{definition}{Definition}[section]

\usepackage{amsthm}
\newtheorem{question}{Question}

\theoremstyle{remark}
\newtheorem{remark}[definition]{Remark}

\theoremstyle{plain}
\newtheorem{proposition}[definition]{Proposition}
\newtheorem{theorem}[definition]{Theorem}
\newtheorem{lemma}[definition]{Lemma}
\newtheorem{result}[definition]{Result}
\newtheorem{corollary}[definition]{Corollary}

\usepackage[usenames]{color}

\newcommand{\OM}{\Omega}

\newcommand\mfd[1]{\mathcal{#1}}

\newcommand{\clos}[1]{\overline{#1}}
\newcommand{\koba}{\mathsf{k}}

\newcommand{\lrarw}{\longrightarrow}

\newcommand{\bdy}{\partial}

\newcommand{\bcdot}{\boldsymbol{\cdot}}

\newcommand{\nat}{\mathbb{N}}

\newcommand{\C}{\mathbb{C}} 
\newcommand{\R}{\mathbb{R}} 
\makeatletter
\newcommand*{\rom}[1]{\expandafter\@slowromancap\romannumeral #1@}
\makeatother

\title[Metric Compactification]{The metric compactification of a Kobayashi hyperbolic complex manifold and a
Denjoy--Wolff Theorem}
\author[Vikramjeet Singh Chandel]{Vikramjeet Singh Chandel}
\address{Department of Mathematics and Statistics, Indian Institute of Technology Kanpur,
Kanpur -- 208 016, India}
\email{vschandel@iitk.ac.in}

\author[Nishith Mandal]{Nishith Mandal}
\address{Department of Mathematics and Statistics, Indian Institute of Technology Kanpur,
Kanpur -- 208 016, India}
\email{nishithm21@iitk.ac.in, nishithmath@gmail.com}

\keywords{Kobayashi hyperbolic domain, metric compactification, 
end compactification, geodesic visibility, Denjoy--Wolff theorem}

\subjclass[2020]{Primary: 32F45, 30D40; Secondary: 32Q45, 53C23.}
		
\begin{document}
	
\begin{abstract}
We study the metric compactification of a Kobayashi hyperbolic complex manifold
\(\mfd{X} \) equipped with the Kobayashi distance \( \koba_{\mfd{X}} \). 
We show that this compactification is genuine\,---\,i.e., \( \mfd{X} \) 
embeds as a dense open subset\,---\,even without completeness of \( \koba_{\mfd{X}} \),
and that it becomes a \emph{good compactification} in the sense of Bharali--Zimmer
when \((\mfd{X}, \koba_\mfd{X}) \) is complete. 
As an application, we obtain a criterion for the
continuous extension of quasi-isometric embeddings from
\( (\mfd{X}, \koba_{\mfd{X}}) \) into visibility domains of complex manifolds.
For a Kobayashi hyperbolic domain \( \Omega \subsetneq \mfd{X} \),
to each boundary point of \( \Omega \) in the end compactification, we associate
a fiber of metric boundary points. This allows the small and big horospheres of 
Abate to be expressed as the intersection and union of horoballs centered at
metric boundary points. We use this to formulate a Wolff-type lemma in terms of horoballs and prove a Denjoy--Wolff theorem
for complete hyperbolic domains satisfying a boundary divergence condition for \(\koba_\mfd{X} \).
Finally, we present necessary and sufficient conditions 
under which the identity map extends continuously
between the metric and end compactifications. 
\end{abstract}

\maketitle

\section{Introduction: Statement of main results}\label{S:intro}

Gromov~\cite{Gromov:1981} introduced a natural embedding of a metric space 
\((X, d)\) into the space of real-valued continuous functions on \(X\).
When \((X, d)\) is a proper geodesic 
metric space, this construction yields a
compactification; see e.g. \cite[Part~II, Chapter~8]{BH:1999}.
By a compactification, we mean:

\begin{definition}\label{Def:compact}
Let \((X, d)\) be a metric space. A pair \((\iota, X^{*})\) is said to be a 
\emph{compactification} of \(X\) if \(X^{*}\) is a sequentially compact Hausdorff topological space, 
\(\iota : X \lrarw X^{*}\) is a homeomorphism onto its image, and \(\iota(X)\) is open and dense in \(X^{*}\).
\end{definition}

Gromov’s construction was later rediscovered 
by Rieffel~\cite[Section~4]{Rieffel:2002}, 
who identified the compactification as the maximal ideal
space of a unital commutative 
\(C^{*}\)-algebra. He referred to this construction as the \emph{metric compactification} and the 
corresponding boundary as the \emph{metric boundary}. These are also known
as the \emph{horofunction compactification} and 
the \emph{horofunction boundary} (see, e.g.,~\cite{AFSG:2024}).
We introduce the metric compactification \((\clos{X}^d, i_d)\) of a metric space $(X, d)$ 
formally in Subsection~\ref{SS:metcomp}, 
along with a discussion of its basic properties.
In this paper, we focus our attention on a Kobayashi hyperbolic complex manifold \(\mfd{X}\), 
equipped with the Kobayashi distance \(\koba_{\mfd{X}}\). Our first result is:

\begin{theorem}\label{thm:metcomptrue}
Let \(\mfd{X}\) be a Kobayashi hyperbolic complex manifold
equipped with the Kobayashi distance \(\koba_{\mfd{X}}\). 
Then \(\big(\clos{\mfd{X}}^{\koba}, i_{\koba}\big)\) is a compactification of \(\mfd{X}\).
\end{theorem}

\noindent Here, \(\clos{\mfd{X}}^{\koba}\) denotes the metric compactification as defined in 
Subsection~\ref{SS:metcomp}, and \(i_{\koba} : \mfd{X} \to \clos{\mfd{X}}^{\koba}\) is the canonical embedding. 
We emphasize that no completeness assumption is imposed on \(\koba_{\mfd{X}}\). 
Indeed, for general hyperbolic complex manifolds, the completeness of the Kobayashi distance 
is often difficult to verify. In the absence of geodesics,
our proof instead relies on the existence of \emph{\((1,\epsilon)\)-quasi-geodesics} 
connecting any two points in \(\mfd{X}\) for arbitrary \(\epsilon > 0\)\,---\,a fact
established by Masanta \cite[Proposition~2.8]{RM:2024}, adapting the argument
used by Bharali--Zimmer \cite[Proposition~4.4]{BZ:2017} for bounded domains.
\smallskip

In the same paper, Bharali--Zimmer also introduced the notion of a 
\emph{good compactification} 
(Definition~\ref{Def:goodcompact}) for geodesic metric spaces that captures the idea that 
geodesics converging to the same ideal boundary point 
must eventually leave every compact subset. Such a condition rules out geodesic looping near the 
boundary and plays a key role in the context of continuous extension of quasi-isometric embeddings; \cite[Section~6.2]{BZ:2017}. Using results of Webster--Winchester \cite{WW:2005},
we observe the following:
\begin{proposition}\label{prop:metgoodcomp}
    Let \((X,d)\) be a proper geodesic metric space. Then 
    the metric compactfication \((\clos{X}^d, i_d)\) is a good compactification of \((X,d)\).
\end{proposition}
\noindent The Hopf--Rinow theorem implies that a complete Kobayashi 
hyperbolic complex manifold \(\mfd{X}\)
is proper and geodesic, and hence the above proposition applies to it.
For a domain \(D\Subset\C^n\) such that \((D, \koba_D)\) admits
a good compactification 
\((D^{*}, \iota)\), Bharali–Zimmer \cite[Theorem~6.5]{BZ:2017} showed that 
any quasi-isometric embedding \(F: D \lrarw \Omega\)\,—\,where \(\Omega\) is a visibility domain\,—\,extends 
continuously to a map from \(D^*\) to \(\clos{\Omega}\).
As an application of Proposition~\ref{prop:metgoodcomp}, we obtain:

\begin{theorem}\label{thm:extquasiisom}
Let \(\mfd{X}\) be a complete Kobayashi hyperbolic complex manifold, and let \(\Omega \subsetneq \mfd{Y}\)
be a Kobayashi hyperbolic domain in a complex manifold $\mfd{Y}$ that satisfies the visibility property. 
Then any continuous quasi-isometric embedding \(F: \mfd{X} \lrarw \Omega\) 
extends to a continuous map from \(\clos{\mfd{X}}^{\koba}\) to \(\clos{\Omega}^{\rm End}\).
\end{theorem}
\noindent Here, visibility is assumed with respect to all 
$(\lambda, \kappa)$-almost-geodesics, $\lambda\geq 1, \kappa>0$ (see 
\cite{RM:2024} for the definition). The proof of the above theorem 
follows with minor modifications from that of \cite[Theorem~6.5]{BZ:2017}.
Theorem~\ref{thm:extquasiisom} implies that the identity map on a complete 
Kobayashi hyperbolic domain $\OM\subsetneq\mfd{X}$ that satisfies geodesic visibility 
property extends continuously from $\clos{\OM}^{\koba}$ to $\clos{\OM}^{\rm End}$. 

\subsection{Dynamics of holomorphic maps}
We now turn to the second theme of this paper, 
which centers on the asymptotic behavior of iterates of holomorphic self-maps.
In particular, we explore analogues of the classical Wolff lemma and Denjoy–Wolff theorem
in the setting of complete Kobayashi hyperbolic domains.
\smallskip

To set the stage, let 
$\xi \in \partial \clos{\mfd{X}}^{\koba} := \clos{\mfd{X}}^{\koba}\setminus i_{\koba}(\mfd{X})$
be a metric boundary point, 
and fix a point $p \in\mfd{X}$. 
Denote by $h_{p,\,\xi}$ the unique horofunction representing $\xi$ that vanishes at $p$ 
(see Definition~\ref{Def:horocompact}). 
The \emph{horoball} with pole $p$, center $\xi$, and radius $R > 0$ is defined as
\[
H_p(\xi,\, R) := \left\{ x \in \mfd{X} \,:\, h_{p,\,\xi}(x) < \tfrac{1}{2} \log R \right\}.
\]
When $\mfd{X}$ is a bounded convex domain in $\C^n$,
the horofunctions $h_{p,\,\xi}$ are quasi-convex, 
and therefore, $H_p(\xi,\, R)$ is a convex open subset of $\mfd{X}$. 
\smallskip

Now let $\Omega \subsetneq \mfd{X}$ be a Kobayashi hyperbolic
domain in a complex manifold $\mfd{X}$. 
We consider its end-point compactification $\overline{\Omega}^{\mathrm{End}}$,
(see Subsection~\ref{SS:endcomp}), 
which we assume to be sequentially compact.
Given a point $p \in \Omega$, following Marco Abate, we define
the \emph{big} and \emph{small horospheres} centered at 
$x \in \partial \overline{\Omega}^{\mathrm{End}}:=\clos{\OM}^{\mathrm{End}}
\setminus \OM$, with pole $p$ and radius $R > 0$, respectively, by:
\begin{align*}
H^b_p(x, R) &:= \Big\{z \in \Omega : \liminf_{w \to x} \left( \koba_\Omega(z, w) - \koba_\Omega(p, w) \right) 
< \tfrac{1}{2} \log R \Big\}, \\
H^s_p(x, R) &:= \Big\{ z \in \Omega : \limsup_{w \to x} \left( \koba_\Omega(z, w) - 
\koba_\Omega(p, w) \right) < \tfrac{1}{2} \log R \Big\}.
\end{align*}
To relate these horospheres to horoballs, we define, for any 
$x \in \partial \clos{\OM}^{\mathrm{End}}$, 
\begin{equation*}
    \mathscr{H}(x) := \left\{ \xi \in \clos{\OM}^{\koba} \setminus i_{\koba}(\OM) \,:\, 
    \exists (x_n) \subset \OM \text{ such that } x_n \to x \text{ and } i_{\koba}(x_n) \to \xi \right\}.
\end{equation*}
%That is, $\mathscr{H}(x)$ denotes the \emph{fiber} of metric boundary points associated to $x$.
The significance of $\mathscr{H}(x)$ is highlighted by Lemma~\ref{lm:smallbighoro}, 
that shows that the big horosphere $H^b_p(x, R)$ is the union of the horoballs $H_p(\xi, R)$ 
for all $\xi \in \mathscr{H}(x)$, while the small horosphere $H^s_p(x, R)$ is their intersection.
\smallskip

For bounded convex domains, not only is $H_p(\xi, R)$ convex, but 
its closure also contains any 
$x \in \partial \clos{\OM}^{\mathrm{End}}$ such that $\xi \in \mathscr{H}(x)$.
This fact provides an explanation for a result due to
Abate--Raissy \cite{MJ:2014} that $\clos{H^b_p(x, R)}$ is star-shaped 
with respect to $x$.
Building on this framework, and drawing motivation
from works of Abate, Karlsson \cite{K:2001},
we present a horoball version 
of Wolff’s lemma adapted to this context.
\begin{theorem}\label{T:wolfftheorem}
    Let $\OM\subsetneq\mfd{X}$ be a complete Kobayashi hyperbolic domain of a complex
    manifold $\mfd{X}$, and suppose that
    \begin{equation}\label{E:kobdiv}
     \lim_{(z, w)\to (x, y)}\koba_\OM(z, w)=\infty \ \ \ 
      \text{for any $x\neq y\in\bdy\OM$}.
    \end{equation}
    Let $f:\OM\lrarw\OM$ be a holomorphic map such that 
    \((f^n)_{n\geq 1}\) is compactly divergent. Then there exist 
    \(x\in\bdy\clos{\OM}^{\rm End}\), and \(\xi\in\mathscr{H}(x)\) such that, for every \(n\in\nat\),
    \begin{equation}\label{E:imhoroballbig}
    f^{n}(H_p(\xi, R))\subset H^b_p(x, R) \ \ \ \text{for all \(R>0\)}.
    \end{equation}
    In particular, for every $n\in\nat$, we also have 
    \begin{equation}\label{E:imsmallbig}
    f^{n}(H^s_p(x, R))\subset H^b_p(x, R) \ \ \ \text{for all \(R>0\)}.
    \end{equation}
    \end{theorem}
\noindent The main feature of the above theorem is the inclusion \eqref{E:imhoroballbig}. 
In general, small horospheres could be empty for sufficiently small values of \(R\), 
even for very nice domains, see e.g., the case of a slit disk in \cite[Example~5.2]{CM:2024}. 
However, horoballs are never empty. 
This distinction allows for a sharper localization of the target set of a 
holomorphic self-map, as shown in Lemma~\ref{L:limitset}. 
This, in turn, plays a crucial role in establishing a Denjoy--Wolff theorem. 
\smallskip

For compactly approximable domains—such as bounded convex domains—Abate 
established a Wolff-type lemma (i.e., \eqref{E:imsmallbig}) for fixed-point 
free holomorphic self-maps by exploiting the interior geometry of the domain. 
Additionally, Abate \cite[Theorem~3.5]{Abate:cds91} proved \eqref{E:imsmallbig}
for general domains \(\OM \Subset\C^n\) under the assumption that
\(\OM\) has a simple boundary, instead of requiring 
condition \eqref{E:kobdiv}. It is known that the simple boundary condition 
implies \eqref{E:kobdiv}; see, for instance, \cite[Appendix~A]{BG:absbound}.
Domains that exhibit visibility with respect to Kobayashi geodesics also 
satisfy \eqref{E:kobdiv}. However, in general,
\eqref{E:kobdiv} appears to be weaker than both the simple boundary
condition and geodesic visibility. We now present the Denjoy--Wolff theorem
alluded to above. 
\begin{theorem}\label{T:DW}
Let $\OM\subsetneq\mfd{X}$ be a complete Kobayashi hyperbolic domain that satisfies condition~\eqref{E:kobdiv}. 
Moreover, suppose that for a fixed point $p \in \OM$ and every $x \in \partial\clos{\OM}^{\mathrm{End}}$, 
\[
\bigcap_{R>0}\overline{H^b_p(x, R)}
\]
is a singleton set. Then, for any holomorphic self-map \(f:\OM \lrarw \OM\), 
one of the following holds:
\begin{itemize}
    \item[$(a)$] For each \(z \in \OM\), the orbit \(\{f^n(z) : n \in \mathbb{N}\}\) is relatively compact in \(\OM\);
    \smallskip
    \item[$(b)$] There exists a point \(y \in \partial\clos{\OM}^{\mathrm{End}}\) such that \(f^n \to y\) 
    in the compact-open topology.
\end{itemize}
\end{theorem}
\noindent By \cite[Proposition~2.4]{BNT:2022}, complete hyperbolic domains with the geodesic visibility property satisfy condition~\eqref{E:kobdiv}. Moreover, they also satisfy 
\[
\overline{H^b_p(x, R)} \cap \partial\clos{\OM}^{\mathrm{End}} = \{x\} \quad \text{for all } R > 0,
\]
as shown in \cite[Proposition~1.2]{CM:2024}. (This proposition is about bounded domains 
but the same proof can be adapted for end compactification.)
Thus, Theorem~\ref{T:DW} applies to such domains.
For strictly convex domains $\OM \Subset \C^n$,
both conditions are verified in \cite[Lemmas~3, 5]{MJ:2014},
so the theorem holds in this case as well.
Finally, for bounded convex domains, a result of Abate \cite[Theorem~2.4.20]{Abate:iteration89} 
implies that part~$(a)$ could be replaced by the fact that $f$ has a fixed point.
\smallskip

The Denjoy–Wolff theorem for strictly convex domains was first proved by Budzyńska \cite{MB:2012}. 
Later, Abate–Raissy \cite{MJ:2014} gave a simpler proof.
Their proof relies heavily on convexity, e.g., small horospheres are always 
nonempty, the target set lies in a linear part of the boundary,
and the intersection of all big horospheres centered at
a point is a singleton. 
Motivated by this, we asked whether the last of these\,---\,being a metric condition\,---\,alone 
suffices for a Denjoy–Wolff-type result under the condition~\eqref{E:kobdiv},
without any geometric assumptions. Our theorem confirms this, potentially
allowing applications to broader settings,
such as hyperbolic Riemann surfaces embedded in ambient Riemann surfaces.
\smallskip

In the same paper, for a bounded convex domain $D \subset \mathbb{C}^d$ and
a holomorphic self-map $f$ of $D$, Abate–Raissy \cite[Lemma~8]{MJ:2014}
established the existence of an invariant horoball under $f$. 
Using this result, one can locate the limit set of $f$\,---\,assuming $f$
is fixed-point free\,---\,purely in terms of horoballs. 
Then, adapting the proof of Theorem~\ref{T:DW} to this setting
yields the following result, which is a stronger version of Theorem~\ref{T:DW} in
the context of convex domains.
 \begin{proposition}\label{prop:bconvex}
Let $\Omega \Subset \mathbb{C}^d$ be a bounded convex
domain that satisfies condition \eqref{E:kobdiv}. 
Given $p \in \Omega$, suppose that for every $\xi \in \partial \overline{\Omega}^{\koba}$,
\begin{equation}
\bigcap_{R > 0} \overline{H_p(\xi, R)}\label{E:inthoroballmet}
\end{equation}
is a singleton. Then for every holomorphic map $f \colon \Omega \to \Omega$, either $f$
has a fixed point in $\Omega$, or there exists
a point $y \in \partial \Omega$ such that $f^n \to y$ uniformly on compact subsets of $\Omega$.
\end{proposition}
\noindent In a recent preprint, Bracci–\"{O}kten \cite[Proposition~4.7]{BO:2025} showed that any bounded convex domain
with simple boundary satisfies the conclusion of Proposition~\ref{prop:bconvex}
if all {\em sequential horospheres} intersect the boundary at exactly one point. 
This also follows from Proposition~\ref{prop:bconvex}, once we observe that their condition implies
\[
\overline{H_p(\xi, R)} \cap \partial \Omega
\]
is a singleton for every $R > 0$. However, in general, 
Proposition~\ref{prop:bconvex} is stronger, 
as it yields the same conclusion under a weaker hypothesis.
The proofs of Theorem~\ref{T:wolfftheorem}, Theorem~\ref{T:DW}, and
Proposition~\ref{prop:bconvex} are presented in Section~\ref{S:horowolffdenjoy}.
\smallskip

When \((\OM, \koba_\OM)\) is Gromov hyperbolic,
the identity map extends continuously from \(\clos{\OM}^{\koba}\)
to the Gromov compactification \(\clos{\OM}^{G}\); see \cite{WW:2005}.
Since condition~\eqref{E:kobdiv} is built into the structure of Gromov boundaries,
adapting the method of Theorem~\ref{T:wolfftheorem} 
could yield a Wolff-type lemma in this setting
and may lead to a Denjoy--Wolff theorem for Gromov hyperbolic spaces.
This approach suggests a sufficient condition for such a theorem in terms of Busemann
horospheres associated with geodesic rays,
partially generalizing a recent result of Bracci--Benini \cite[Theorem~1.4]{BB:2024}.
We hope to develop this direction in our forthcoming work.
\smallskip

We now move to the third theme of this paper:
namely the continuous extension of the identity 
map from $\clos{\OM}^{\koba}$ to $\clos{\OM}^{\rm End}$ and vice-versa.
In this direction, we have the following result.
\begin{theorem}\label{thm:mettoend}
    Let $\OM\subsetneq\mfd{X}$ be a Kobayashi hyperbolic domain 
    in a complex manifold $\mfd{X}$. 
    Then there exists a continuous surjective map
    $F:\clos{\OM}^{\koba}\lrarw\clos{\OM}^{\rm End}$ such that 
\begin{equation}\label{E:extcano}
F\circ i_{\koba}={i}_\OM \ \ \ \text{on $\OM$}
\end{equation}
if and only if $\mathscr{H}(x)\cap\mathscr{H}(y)=\emptyset$ for all
$x\neq y\in\bdy\clos{\OM}^{\rm End}$. 
\end{theorem}
\noindent The proof of the above theorem is given in Section~\ref{SS:proofmettoend} in which 
we also provide a sufficient condition (Proposition~\ref{prop:suff}) for the existence of \(F\), namely: 
\[
\bigcap_{R>0} \overline{H^b_p(x, R)}=\{x\}
\]
for each \(x \in \partial\clos{\OM}^{\mathrm{End}}\), 
which holds, for instance, in domains with geodesic visibility property
or bounded strictly convex domains. A necessary condition, for the existence of $F$ as above, 
is that all geodesic rays land at the boundary, see Proposition~\ref{prop:necmettoend}.  
In the case of Gromov hyperbolicity, Bharali--Zimmer \cite[Theorem~1.11]{BZ:2023}
showed that extendability of the identity map from $\clos{\OM}^G$ to $\clos{\OM}^{\mathrm{End}}$
implies geodesic visibility, 
and the converse follows from ideas in \cite[Theorem~3.3]{BNT:2022}.
Thus, visibility characterizes extendability in the Gromov setting,
whereas for the metric compactification, the corresponding condition appears weaker
than visibility (cf. Theorem~\ref{thm:extquasiisom}). 
\smallskip

For the converse, the map \(i_{\koba} : \clos{\OM}^{\mathrm{End}} \lrarw \clos{\OM}^{\koba}\) extends continuously if and only if \(\mathscr{H}(x)\) is a singleton for each \(x \in \partial \clos{\OM}^{\mathrm{End}}\), which is equivalent to the existence of the limit
\[
\lim_{w \to x} \big( \koba_\OM(z, w) - \koba_\OM(o, w) \big)
\]
for every \(z \in \OM\).
The equivalence between the continuous extension and
the existence of the above limit is well known
in the context of bounded domains in $\C^n$.
For completeness, we record this fact in its most general form as Theorem~\ref{T:endtomet}.
\medskip

\noindent{\bf Concluding Remarks.}
Recall that all strictly convex domains satisfy the Denjoy--Wolff theorem.
Moreover, as observed above, in any strictly convex domain,
every geodesic ray lands at a boundary point. 
Motivated by a recent work of Bracci--Benini \cite{BB:2024},
we pose the following question:
\begin{question}
Does there exist $\OM\Subset\C^d$, a bounded strictly convex domain, 
that does not satisfy the geodesic visibility property?
\end{question}
\noindent This question is also related to Question~3.10 in the
recent preprint by Bracci--\"{O}kten \cite{BO:2025},
where the authors ask whether there exist bounded convex domains that are not visibility domains,
but for which all geodesic rays land.
\smallskip

We also observe that all strictly convex domains satisfy
a logarithmic estimate for the Kobayashi distance, as established in
\cite[Proposition~4.1]{BNT:2022}.
Therefore, any potential counterexample to the visibility
property within the class of strictly convex domains must
have a boundary that is not Dini-smooth.
In particular, such a domain would need to exhibit relatively low boundary regularity.

\section{Preliminaries}\label{S:prelims}
In this section, we collect certain preliminary notions, definitions, and auxiliary lemmas 
that will be useful in the upcoming sections.

\subsection{End Compactification}\label{SS:endcomp}
Throughout this paper, we consider a subdomain $\OM \subsetneq \mfd{X}$ in a complex manifold $\mfd{X}$. 
The {\em end compactification} $\clos{\OM}^{\rm End}$ of $\clos{\OM}$ in $\mfd{X}$ is defined as the set
\[
\clos{\OM}^{\rm End} := \clos{\OM} \,\sqcup\, \mathcal{E},
\]
where $\mathcal{E}$ is the set of all ends of $\clos{\OM}$. An {\em end} is defined as follows:
fix a compact exhaustion $(K_j)_{j\geq 1}$ of $\clos{\OM}$ such that $K_j \subset {\rm int}(K_{j+1})$ for each $j$.
An end $e$ is a sequence $(F_j)_{j\geq 1}$ where each $F_j$ is a connected component of 
$\clos{\OM} \setminus K_j$, and $F_{j+1} \subset F_j$ for all $j \geq 1$. 
Note that if $\OM \Subset\mfd{X}$, then $\mathcal{E} = \emptyset$.
We now describe a topology on $\clos{\OM}^{\rm End}$. For $x \in \clos{\OM}$, 
a neighborhood basis is given by all neighborhoods of $x$ in $\mfd{X}$ intersected with $\clos{\OM}$. 
For an end $e = (F_j)_{j \geq 1}$, the neighborhood basis is given by the family $(\widehat{F}_j)_{j \geq 1}$ defined as
\[
\widehat{F}_j :=
F_j \,\sqcup\, \left\{f \in \mathcal{E} \,\big|\, f = (G_\nu)_{\nu \geq 1} \text{ with } G_\nu = F_\nu \ \text{for } \nu = 1, \ldots, j \right\}.
\]
With this topology, $\clos{\OM}^{\rm End}$ becomes a first-countable Hausdorff topological space, and the resulting topology 
does not depend on the choice of compact exhaustion. In general, $\clos{\OM}^{\rm End}$ need not be compact 
or even sequentially compact. However, when $\clos{\OM}$ is locally connected, compactness of $\clos{\OM}^{\rm End}$ follows. 
A sufficient condition for sequential compactness is provided in \cite[Result~2.1]{CGMS:2024}; 
in particular, any form of visibility property ensures sequential compactness. 
Throughout this paper, we shall assume that $\clos{\OM}^{\rm End}$ is sequentially compact. We now present a lemma that is needed later. In 
what follows, $h$ will denote a hermitian metric on $\mfd{X}$, and 
$d_h$ will be the distance induced by $h$.

\begin{lemma}\label{lm:nbhdunbseq}
Let $\OM \subsetneq \mfd{X}$ be a domain in a complex manifold $\mfd{X}$ such that $\clos{\OM}^{\mathrm{End}}$ is sequentially compact. Let $(x_n)_{n \geq 1} \subset \clos{\OM}$ be a sequence that is not relatively compact in $\mfd{X}$. Suppose there exists a compact set $K \subset \clos{\OM}$ and a connected component $F_0$ of $\clos{\OM} \setminus K$ such that $x_n \in F_0$ for all $n$. Then $x_n \in \mathrm{int}(F_0)$ for all but finitely many $n$.
\end{lemma}

\begin{proof}
Let $\{F_0\} \cup \{F_\lambda : \lambda \in \Lambda\}$ be the
collection of all connected components of $\clos{\OM} \setminus K$, where each $F_\lambda$ is closed in $\clos{\OM} \setminus K$.
\smallskip

\noindent\textbf{Claim.} If $x_n \notin \mathrm{int}(F_0)$ for some $n$, then for any $r > 0$, the ball $B_{d_h}(x_n, r)$ intersects infinitely many components $F_\lambda \neq F_0$.
\smallskip

\noindent Suppose not. Then there exists $r > 0$ such that $B_{d_h}(x_n, r)$ intersects
only finitely many components other than $F_0$, say $F_{\lambda_1}, \ldots, F_{\lambda_k}$.
Now, for every $\epsilon \in (0, r)$, the ball $B_{d_h}(x_n, \epsilon)$
intersects at least one of these $F_{\lambda_i}$'s, implying that $x_n$
is a limit point of $F_{\lambda_1} \cup \cdots \cup F_{\lambda_k}$,
which is a finite union of closed sets. Hence, $x_n$ must belong to one of them,
contradicting the assumption that $x_n \in F_0$. This proves the claim.
\smallskip

Now suppose that $x_n \notin \mathrm{int}(F_0)$ for infinitely many $n$.
Without loss of generality, assume $x_n \notin \mathrm{int}(F_0)$ for all $n$.
Fix $r > 0$. For each $n$, use the claim to choose a point
$w_n \in B_{d_h}(x_n, r) \cap F_{\lambda_n}$ such that
$\lambda_n \in \Lambda$ and $F_{\lambda_n} \notin \{F_{\lambda_1}, \ldots, F_{\lambda_{n-1}}\}$.
This is possible since each $B_{d_h}(x_n, r)$ intersects infinitely many components.
Then $(w_n)$ is not relatively compact in $\mfd{X}$.
We now show that $(w_n)$ has no convergent
subsequence in $\clos{\OM}^{\mathrm{End}}$. 
Let $(K_j)_{j \geq 1}$ be a compact exhaustion
of $\clos{\OM}$ with $K_1 = K$. 
Since each $w_n$ lies in a different component of $\clos{\OM} \setminus K_1$,
no end can capture infinitely many $w_n$, and there is no subsequence
converging to a point in $\clos{\OM}^{\mathrm{End}}$. 
This contradicts the sequential compactness of $\clos{\OM}^{\mathrm{End}}$, completing the proof.
\end{proof}

\subsection{Metric compactification}\label{SS:metcomp}
In this subsection, we review the construction of the metric compactification
for a general metric space. 
To facilitate its use later, we also include a few foundational results\,---\,both statements
and proofs\,---\,that are standard but not always readily referenced.
\smallskip

Given a metric space $(X,d)$, let $C(X)$ be the space of all real-valued continuous
functions on $X$. We equip $C(X)$ with the topology of
local uniform convergence. Let $C_*(X)$  be the quotient space 
of $C(X)$ by the subspace of constant functions.
Consider the canonical projection $\pi:C(X)\lrarw C_*(X)$ defined by 
\[ 
  \pi(f):=[\,f\,]:=\left\{f+c:c\in\R\right\} .
  \]
We equip $C_*(X)$ with the quotient topology. In this topology, $[\,f_n\,]\to [\,0\,]$
if and only if there exist $(a_n)\subset\R$ such that $f_n+a_n\to 0$ 
locally uniformly on $X$. 
Given a point $p\in X$, let $C_p(X)$ be the set of all 
real-valued continuous functions that vanish at $p$. 
Consider $T_p:C_*(X)\lrarw C_p(X)$ given by $T_p([\,f\,]):=f-f(p)$.
It follows that $T_p$ is well-defined, and moreover we have the following lemma, 
that says that $C_{*}(X)$ can be identified with $C_p(X)$.
\begin{lemma}\label{L:realization_C_*(X)}
    Let $p\in X$ be given. Then the map $T_p:C_{*}(X)\lrarw C_{p}(X)$ given by 
    $T_p([f]):=f-f(p)$ is a homeomorphism.
    \end{lemma}
    \noindent The proof is elementary and is therefore omitted. 
    \smallskip

Consider the map $i_d:X\lrarw C_*(X)$ defined by $i_d(x):=[\,d_x(\bcdot)\,]$, where $d_x:X\to\R$
is the continuous function defined by $d_x(z):=d(x,z)$. When we fix a point $p\in X$,
$T_p\circ i_d(x)$ is the continuous function $d(x, \cdot)-d(x, p)$ for each $x\in X$. 

\begin{lemma}\label{L:embedding}
Let $(X,d)$ be a metric space. Then the map $i_d:X\lrarw C_*(X)$
is an injective continuous map. 
\end{lemma}
\begin{proof}
    It is sufficient to show that for a fixed $p\in X$, $T_p\circ i_d:X\lrarw C_p(X)$ is one-one
    and continuous. Suppose $T_p\circ i_d(x)=T_p\circ i_d(y)$ for some $x,y\in X$. Therefore 
    $d_x(\bcdot)-d_x(p)=d_y(\bcdot)-d_y(p)$ on $X$. By putting $x$ and $y$ separately in the last 
    equation, we get $d(x,y)=0$, i.e. $x=y$. Hence $T_p\circ i_d$ is one-one.
    Now we show that $T_p\circ i_d$ is continuous. Let $(x_n)$ be a sequence convergent to $x$
    in $X$. Then
    \begin{align*}
        |T_p\circ i_d(x_n)-T_p\circ i_d(x)|=|d_{x_n}(\bcdot)-d(p,x_n)-d_x(\bcdot)+d(p,x)|
        \leq 2\,d(x_n,x).
    \end{align*} 
    Therefore, $T_p\circ i_d(x_n)$ converges to $T_p\circ i_d(x)$ uniformly in $C_p(X)$. Hence $T_p\circ i_d$ is continuous, from which the result follows.
\end{proof}

\begin{definition}\label{Def:horocompact}
    Let $(X,d)$ be a metric space. The {\em metric compactification} 
    denoted by $\clos{X}^d$ of $X$ is the closure of $i_d(X)$ in $C_*(X)$ and the
    {\em metric boundary} of $X$ is the set 
    \[
    \partial\clos{X}^d:=\clos{X}^d\setminus i_d(X).
    \]
Given $\xi\in\partial\clos{X}^{d}$, there exists a function $h_\xi\in C(X)$,
unique up to an additive constant, 
such that $[\,h_\xi\,]=\xi$. The function $h_\xi$ is called a {\em horofunction}
corresponding to $\xi$. In particular, if we fix a point $p\in X$, then there is a unique 
such function $h_{p,\,\xi}$ which vanishes at the point $p$. Given $R>0$, the set
\[
H_p(\xi, R):=\left\{x\in X\,:\,h_{p,\,\xi}(x)<\tfrac{1}{2}\log R\right\}
\]
is called the horoball of radius $R$ centered at $\xi$ with pole $p$. 
\end{definition}
\begin{lemma}\label{L:compact set}
    Let $(X,d)$ be a metric space. Then the 
    metric compactification $\clos{X}^d$ of $X$ is compact.
    \end{lemma}
    \begin{proof}
        Fix a point $p\in X$. Using Lemma~\ref{L:realization_C_*(X)}, we can identify $\clos{X}^d$
        as the closure of the family of functions $\{T_p\circ i_d(x):x\in X \}$. Note that
        \[
         |T_p\circ i_d(x)(y_1)-T_p\circ i_d(x)(y_2)|\leq d(y_1,y_2)\ \ \ \forall x,y_1,y_2\in X
         \]
        Therefore, the family $\{T_p\circ i_d(x):x\in X \}$ is equicontinuous. 
        Also note that for any $y\in X$
        \[ 
          |T_p\circ i_d(x)(y)|\leq d(p,y).
          \]
        So the family $\{T_p\circ i_d(x):x\in X \}$ is uniformly bounded.
        Since $X$ is locally compact, by Ascoli's Theorem \cite[Theorem~47.1]{{JM:2000}}, the 
        family $\{T_p\circ i_d(x):x\in X \}$ is relatively compact in $C_p(X)$.
        Hence $\clos{X}^d$ is compact.
    \end{proof}

    %When $(X, d)$ is nice enough, e.g. separable, the space $\clos{X}^d$ is metrizable also.
    Finally, given $\epsilon>0$, we recall that a map $\sigma: I\lrarw (X,d)$ is called a $(1, \epsilon)$-quasi-geodesic if 
    \[
    |t_1-t_2|-\epsilon\leq d(\sigma(t_1), \sigma(t_2))\leq |t_1-t_2|+\epsilon \quad\forall t_1, t_2\in I,
    \]
    where $I\subset\R$ is an interval. The following lemma follows easily from the definition of $(1, \epsilon)$-quasi-geodesic. 
    \begin{lemma}\label{L:quasigeod}
    Let $(X, d)$ be a metric space and let 
    $x, y, z\in X$ be points that lie on 
    $(1, \epsilon)$-quasi-geodesic. Then 
    \[
    d(x, z)\leq d(x, y)+d(y, z)\leq d(x, z)+3\epsilon
    \]
    \end{lemma}
    \smallskip
    
\section{Metric compactification of a hyperbolic complex manifold}\label{S:Metric compactification}
In this section, we present the proofs of
Theorem~\ref{thm:metcomptrue},
Proposition~\ref{prop:metgoodcomp} and 
Theorem~\ref{thm:extquasiisom}. 
For a complex manifold $\mfd{X}$, 
given $p, z\in \mfd{X}$, in what follows,
we shall use the notation $\phi_{z,\,p}(\bcdot)$ to denote the function
$T_p\circ i_{\koba}(z)=\koba_{\mfd{X}}(\bcdot, z)-\koba_{\mfd{X}}(p, z)$.
We shall also denote $\phi_{z,\,p}$ simply by $\phi_z$ if 
$p$ is fixed throughout the discussion. 
The following result was already mentioned in the introduction,
and is needed in our proof of Theorem~\ref{thm:metcomptrue}. 
\smallskip

\begin{result}[{paraphrasing \cite[Proposition~2.8]{RM:2024}}]\label{Res:exisquasi}
Let $\mfd{X}$ be a Kobayashi hyperbolic complex 
manifold equipped with the Kobayashi distance 
$\koba_{\mfd{X}}$. Then, given $\epsilon>0$, 
for any two points $z, w\in\mfd{X}$, there exists 
a $(1, \epsilon)$-quasi-geodesic connecting them. 
\end{result}

We are now ready to present
\begin{proof}[The proof of Theorem~\ref{thm:metcomptrue}]
    Fix a point $p\in\mfd{X}$. Appealing to Lemma~\ref{L:compact set}
    and Lemma~\ref{L:embedding}, it is sufficient to show that $T_p\circ i_{\koba}$ is an open map 
    onto its image in $C_p(\mfd{X})$. Let $U$ be an open set in $\mfd{X}$. 
    Observe that
    \[
    T_p\circ i_{\koba}(U)=\Big\{\koba_{\mfd{X}}(\bcdot,z)-\koba_{\mfd{X}}(p, z):z\in U\Big\}.
    \]
    We shall show the complement of 
    $T_p\circ i_{\koba}(U)$ in $T_p\circ i_{\koba}(\mfd{X})$ is closed. 
    Consider the sequence given by 
    $\phi_{z_n}(\bcdot):=\koba_{\mfd{X}}(\bcdot,z_n)-\koba_{\mfd{X}}(p, z_n)$, 
    where $(z_n)\subset\mfd{X}\setminus U$, that converges locally uniformly to $\phi_w(\bcdot)$
    for some $w\in\mfd{X}$. We now show that $w\in\mfd{X}\setminus U$. 
    Two cases arise:
    \smallskip
    
    \noindent\textbf{Case 1:} The sequence $(z_n)$ has a subsequence that converges
    to $z_0\in\mfd{X}$. Assume, by reindexing if necessary,
    that $z_n\to z_0$. Then, as observed
    in Lemma~\ref{L:embedding}, $\phi_{z_n}(\bcdot)$ converges locally uniformly to $\phi_{z_0}(\bcdot)$.
    By injectivity of $T_p\circ i_{\koba}$, we get $w=z_0$. Since $\mfd{X}\setminus U$ is closed, $w=z_0\in\mfd{X}\setminus U$ and we are done.
    \smallskip
    
    \noindent{\bf Case 2:} The sequence $(z_n)$ has no limit point
    in $\mfd{X}$. Since $\big(\mfd{X},\koba_{\mfd{X}}\big)$
    is locally compact, there exists an $r>0$ such that 
    $\overline{B_\koba(w,r)}$ is compact in $\big(\mfd{X},\koba_{\mfd{X}}\big)$.
    Here, ${B_\koba(w,r)}=\{z\in\mfd{X}:\koba_\mfd{X}(z, w)<r\}$. 
    The sequence $(z_n)$ has only finitely many points in the closed ball $\overline{B_\koba(w,r)}$.
    Therefore, we may assume that for each $n\geq 1$,
    $\koba_{\mfd{X}}(z_n,w)>r$.
    Let $\epsilon=r/3$, and choose 
    $\sigma_n$, a $(1,\epsilon)$-quasi-geodesic
    connecting $z_n$ and $w$, by applying Result~\ref{Res:exisquasi}. 
    Let $w_n\in\partial B_{\koba}(w,r)$ be such that $w_n\in{\rm Image}(\sigma_n)$.
    Since $\partial B_{\koba}(w,r)$ is compact, by passing to a subsequence and 
    reindexing, we can assume that $w_n\to w_r$ for some 
    $w_r\in\partial B_{\koba}(w,r)$. Now, since $\phi_{z_n}\to\phi_{w}$ locally 
    uniformly, we have
    \begin{equation}\label{E:limphizn}
     \lim_{n\to\infty}\phi_{z_n}(w_n)=\phi_w(w_r)=
     \koba_{\mfd{X}}(w,w_r)-\koba_{\mfd{X}}(w,p)=r-\koba_{\mfd{X}}(w,p). 
    \end{equation}
    On the other hand, for each $n\in\nat$, we have 
    $\phi_{z_n}(w_n)=\koba_{\mfd{X}}(z_n,w_n)-\koba_{\mfd{X}}(z_n,p)$. 
    Since the points $w$, $w_n$ and $z_n$ lie on a $(1,\epsilon)$-quasi-geodesic, 
    by Lemma~\ref{L:quasigeod}, for each $n\in\nat$, we have
    \begin{align*}
         \koba_{\mfd{X}}(w,w_n)+\koba_{\mfd{X}}(w_n,z_n)\leq \koba_{\mfd{X}}(w,z_n)+3\epsilon
          \implies 
          \koba_{\mfd{X}}(w_n,z_n)\leq \koba_{\mfd{X}}(w,z_n)-r+3\epsilon \ \ \ 
          \forall n\in\nat. 
    \end{align*}
    Therefore, for every $n\in\nat$, 
    \[
      \koba_{\mfd{X}}(w_n,z_n)-\koba_{\mfd{X}}(z_n,p)\leq 
      \koba_{\mfd{X}}(w,z_n)-\koba_{\mfd{X}}(z_n,p)-r+3\epsilon
      \implies \phi_{z_n}(w_n)\leq \phi_{z_n}(w)-r+3\epsilon\] for all $n$. 
    Taking $n\to\infty$, and using \eqref{E:limphizn}, we get
    \[ 
     r-\koba_{\mfd{X}}(w,p)\leq-\koba_{\mfd{X}}(w,p)-r+3\epsilon
     \]
    which implies that $2r\leq3\epsilon=r$, which is absurd. Hence this case does not arise.
\end{proof}
\begin{remark}\label{rm:charmetb}
Note that it follows from the proof of Theorem~\ref{thm:metcomptrue}
that $\xi\in\bdy\clos{\mfd{X}}^{\koba}:=\clos{\mfd{X}}^{\koba}\setminus i_{\koba}(\mfd{X})$, i.e., $\xi$ is a
metric boundary point
if and only if there exists a sequence $(z_\nu)_{\nu\geq 1}\subset\mfd{X}$
such that $i_{\koba}(z_\nu)\to\xi$ and $(z_\nu)_{\nu\geq 1}$ does not have a limit point
in $\mfd{X}$. 
\smallskip

The proof of the above theorem could be adapted to any 
locally compact metric space in which every pair of points can be joined by 
a $(1, \epsilon)$-quasi-geodesic for any $\epsilon>0$. 
\end{remark} 

Focusing our attention on a proper geodesic metric space 
\((X, d)\), where \((\clos{X}^{d}, i_d)\) forms a
compactification, we now show that it is, in fact, a \emph{good compactification} in the
sense of Bharali--Zimmer \cite{BZ:2017}, defined as follows.

\begin{definition}\label{Def:goodcompact}
Suppose $(\iota, X^{*})$ is a compactification of a geodesic metric space $(X, d)$. We say 
$(\iota, X^{*})$is a {\em good compactification} if for all sequences
$\sigma_n:[a_n, b_n]\lrarw X$ of geodesics with the property
\[
\lim_{n\to\infty}\iota(\sigma_n(a_n))=\lim_{n\to\infty}\iota(\sigma_n(b_n))
\in{X}^{*}\setminus\iota(X)
\]
we have 
\[
\liminf_{n\to\infty}d(o, \sigma_n)=\infty
\]
for any $o\in X$.
\end{definition}

We require a few foundational lemmas due to Webster–Winchester \cite{WW:2005} 
to establish our main result. Since their proofs are elementary and 
the results fundamental, we include self-contained proofs adopted to our setting.
The first lemma \cite[Lemma~4.1]{WW:2005} essentially says 
that horoballs of arbitrarily small radius are nonempty. 
In our work \cite[Lemma~3.2]{CM:2024}, 
we independently observed that the big horospheres are always nonempty.
The proof of the lemma below follows the idea presented in the latter work.
\begin{lemma}\label{lm:horoneg}
Let $(X,d)$ be a proper, geodesic metric space and let $h\in C(X)$
be a horofunction. Then for any $R<0$, there exists $x\in X$ 
such that $h(x)<R$.
\begin{proof}
    Fix a point $p\in X$, then there exists a sequence 
    $(x_n)_{n\geq 1}\subset X$ such that
    $(d(\cdot, x_n)-d(p, x_n))_{n\geq 1}$ converges locally uniformly to
    $h(\cdot)-h(p)$.
    Note, as noted in Remark~\ref{rm:charmetb},
    $(x_n)_{n\geq 1}$ cannot have a limit point in $X$. Since $X$ is a proper metric 
    space, this implies $\lim_{n\to\infty}d(p,x_n)=\infty$.
    For $s>0$, consider the ball $B_d(p,s)$. Observe, we may assume that 
    $d(p,x_n)>s$ for all $n\geq 1$. 
    For each $n$, choose $b_n\in\partial B_d(p,s)$ such that 
    \begin{equation}\label{E:splitd(p,xn)}
    d(p,x_n)=d(p,b_n)+d(b_n,x_n)\quad \forall n\in\nat.
    \end{equation}
    By passing to a subsequence, 
    we may assume $b_n\to b_s$ as $n\to\infty$, for some $b_s\in\partial B_d(p,s)$.
    Therefore, given $\epsilon >0$, there exists $n_0\in\nat$ such that $d(b_s,b_n)<\epsilon$
    for all $n\geq n_0$. Therefore, 
    \begin{equation}\label{E:inequality}
        d(b_s,x_n)\leq d(b_s,b_n)+d(b_n,x_n)<\epsilon +d(b_n,x_n)\quad \forall n\geq n_0.
    \end{equation}
    Using \eqref{E:splitd(p,xn)} and \eqref{E:inequality},we get 
    \[
    d(b_s,x_n)-d(p,x_n)<-s+\epsilon
    \quad\forall n\geq n_0.
    \]
    Taking $n\to\infty$, we get 
    $h(b_s)-h(p)\leq -s+\epsilon$ for all $s>0$, $\epsilon >0$.
    Therefore, $h(b_s)<-s/2+h(p)$ for all $s>0$. Now by choosing $s>2(h(p)-2R)$, we get the result.
\end{proof}
\end{lemma}
Given points $x, y\in (X, d)$, the Gromov product of 
$x, y$ with respect to a point $p\in X$, denoted by $\langle x, y\rangle_p$, is 
defined by
\[ 
\langle x, y\rangle_p:=\frac{d(x,p)+d(y,p)-d(x,y)}{2}. 
\]
\noindent We now present the next lemma. 
\begin{lemma}
Let $(X, d)$ be a metric space. Fix a point $p\in X$, then for any $x, y, z\in X$, we have
\[
\langle x, y\rangle_p\geq -\frac{\big(T_p\circ i_d(x)(z)+T_p\circ i_d(y)(z)\big)}{2},
\]
where $T_p$ is as in Lemma~\ref{L:realization_C_*(X)}. 
%Moreover the above inequality is equality if and only if $z\in[x,y]$. 
\end{lemma}
\begin{proof}
Let $z$ be any point in $X$ then 
\begin{align*}
2\langle x,y\rangle_p& =d(x,p)+d(y,p)-d(x,y)
                  \geq d(x,p)+d(y,p)-d(x,z)-d(z,y)\\
                  &=-\big(d(x,z)-d(x,p)\big)-\big(d(z,y)-d(y,p)\big)\\
                  &=-T_p\circ i_d(x)(z)-T_p\circ i_d(y)(z),
\end{align*}
from which the result follows. 
\end{proof}

\begin{lemma}
    Let $(X,d)$ be a proper, geodesic metric space. Let $(x_n)_{n\geq 1}\subset X$
    be a sequence such that $(i_d(x_n))_{n\geq 1}\subset C_*(X)$
    converges to a horofunction. Then, for any $p\in X$, we have
    \[ 
     \lim_{n\to\infty}\big<x_n,x_m\big>_p=\infty.
    \]
\end{lemma}
\begin{proof}
    Fix $p\in X$ then $T_p\circ i_d(x_n)\lrarw h\in C(X)$ locally uniformly
    and $h$ is a horofunction. Let $M>0$ be given then by 
    Lemma~\ref{lm:horoneg}, we can find $z\in X$ such that $h(z)<-M$.
    Since $T_p\circ i_d(x_n)\lrarw h\in C(X,\R)$ locally uniformly, there 
    exists $N_0\in\nat$ such that $T_p\circ i_d(x_n)(z)<{-M}/{2}$ for all
    $n\geq N_0$.
    Therefore, for $n,m\geq N_0$, we get
    \begin{align*}
    \big<x_n,x_m\big>_p \geq -\frac{\big(T_p\circ i_d(x_n)(z)+T_p\circ i_d(x_m)(z)\big)}{2}
                       \geq \frac{M}{2}
    \end{align*}
    Hence the result.
\end{proof}
\noindent The next lemma is in similar vein and easily follows from the above. So we only state the lemma.
\begin{lemma}\label{lm:gromprodinf}
    Let $(X,d)$ be a proper, geodesic metric space.
    Let $(x_n),(y_n)\subset X$ be two sequences such that 
    $(i_d(x_n))$ and $(i_d(y_n))$ converge to the same horofunction then 
    \[
    \lim_{n,m\to\infty}\big<x_n,y_m\big>_p=\infty.
    \]
\end{lemma}
We now have all the tools to present

\begin{proof}[The proof of Proposition~\ref{prop:metgoodcomp}]
To show that $(\clos{X}^d,i_d)$ is a good compactification of $(X,d)$.
Let $\sigma_n:[a_n,b_n]\lrarw(X,d)$ be a sequence of geodesics such that
$\lim_{n\to\infty}i_d(\sigma_n(a_n))=
\lim_{n\to\infty}i_d(\sigma_n(b_n))\in\clos{X}^d\setminus i_d(X)$.
Let us write $x_n=\sigma_n(a_n)$ and $y_n=\sigma_n(b_n)$. Fix a point
$p\in X$. 
\smallskip

    \noindent{\bf Claim.} $d\big(p,\sigma_n([a_n,b_n])\big)\geq \langle x_n, y_n\rangle_p$. 
    \smallskip
    
    \noindent To see the claim, let $z_n\in\sigma_n([a_n, b_n])$ be such that
    $d(p, z_n)=d\big(p,\sigma_n([a_n,b_n])\big)$. Then, using the triangle inequality and that
    $\sigma_n$ is a geodesic, we have
    \begin{align*}
        2 d\big(p, \sigma_n([a_n, b_n])\big)&=2 d(p, z_n)\\
        &\geq (d(p, x_n)-d(x_n, z_n))+(d(p, y_n)-d(y_n, z_n))\\
        &=d(p, x_n)+d(p, y_n)-d(x_n, y_n),
    \end{align*}
    from which the claim follows.
    \smallskip

    By Lemma~\ref{lm:gromprodinf}, we have
    $\lim_{n\to\infty}\big<x_n,y_n\big>_p=\infty$, and therefore by the claim
    above, 
    \[\liminf_{n\to\infty}d(p,\sigma_n[a_n,b_n])=\infty.
    \]
    This establishes the result 
    that $(\clos{X}^d,i_d)$ is a good compactification of $X$.
\end{proof}
We now have all the ingredients to present

\begin{proof}[The proof of Theorem~\ref{thm:extquasiisom}]
Note by Proposition~\ref{prop:metgoodcomp}, $(\clos{\mfd{X}}^{\koba}, i_{\koba})$ is a good compactification
of $\mfd{X}$. Moreover, it follows from Theorem~\cite[Theorem~1.16]{RM:2024} that $\OM$ is a 
visibility domain subordinate to $\clos{\OM}^{\rm End}$. 
Using exactly the same arguments as in the proof of \cite[Theorem~6.5]{BZ:2017}, it 
follows that for every $\xi\in\bdy\clos{\mfd{X}}^{\koba}$, there exists a $x\in\bdy\clos{\OM}^{\rm End}$
such that
\[
\lim_{i_{\koba}(z)\to\xi}F(z)=x.
\]
Now consider the map $\widetilde{F}:\clos{\mfd{X}}^{\koba}\lrarw\clos{\OM}^{\rm End}$ defined by 
    \[
    \widetilde{F}(\xi):=
    \begin{cases}
        F(\xi), & \ \ \ \xi\in\mfd{X},\\ 
        \lim_{i_{\koba}(z)\to\xi}F(z), & \ \ \ \text{if $\xi\in\bdy\clos{\mfd{X}}^{\koba}$.}
    \end{cases}
    \]
It is now routine to check that $\widetilde{F}$ is continuous. 
\end{proof}
\section{Horoballs, a Wolff-type lemma and a Denjoy--Wolff 
Theorem}\label{S:horowolffdenjoy}
In this section, we shall present the proofs of our main results
regarding the dynamics of a holomorphic 
self-map of a Kobayashi hyperbolic domain of a complex manifold. 
But first we begin with describing the exact relation beween horoballs 
and the small and big horospheres as mentioned in the introduction.

\subsection{Horoballs, small and big horospheres}\label{SS:horobsmallbig}
Let $\OM\subsetneq\mfd{X}$ be a Kobayashi hyperbolic domain of a complex
manifold $\mfd{X}$. We shall denote by
$(\clos{\OM}^{\koba}, i_{\koba})$ the metric 
compactification of the distance space $(\OM, \koba_\OM)$. Recall that by Theorem~\ref{thm:metcomptrue}, 
$\clos{\OM}^{\koba}$ is a compactification of $\OM$ as defined in Definition~\ref{Def:compact}.
We shall also consider $\clos{\OM}^{\rm End}$,
the end-compactification of $\OM$ as introduced
in Subsection~\ref{SS:endcomp}. As remarked in that subsection, we will always assume that 
$\clos{\OM}^{\rm End}$ is sequentially compact. 
\smallskip
Given $x\in\partial\clos{\OM}^{\rm End}:=\clos{\OM}^{\rm End}\setminus\OM$,
recall
\begin{equation*}
    \mathscr{H}(x):=\big\{\xi\in \clos{\OM}^{\koba}\setminus i_{\koba}(\OM)\,:\,\text{$\exists
    (x_n)_{n\geq 1}\subset\OM$ such that $x_n\to x$ and $i_{\koba}(x_n)\to\xi$}\big\}.
\end{equation*}
Our first lemma precisely describes the relation
between the big, small horospheres
and the horoballs centered at metric boundary points. 
\begin{lemma}\label{lm:smallbighoro}
Given $p\in\OM$, $R>0$ and $x\in\partial\clos{\OM}^{\rm End}$, we have
\[
H^b_p(x, R)=\bigcup_{\xi\in\mathscr{H}(x)}H_p(\xi, R), \quad\text{and} \quad
H^s_p(x, R)=\bigcap_{\xi\in\mathscr{H}(x)}H_p(\xi, R),
\]
where $H^b_p(x, R), H^s_p(x, R)$ 
are the big and small horospheres with center
$x$, pole $p$ and radius $R$
respectively. 
\end{lemma}
\begin{proof}
Recall that 
\begin{equation*}\label{E:bighoro}
H^b_p(x, R)=\left\{z\in\OM\,:\,\liminf_{w\to x}\big(\koba_\OM(z, w)-\koba_\OM(p, w)\big)
<\tfrac{1}{2}\log R\right\}.
\end{equation*}
Let $\xi\in\mathscr{H}(x)$ and let $h_{p,\,\xi}$ be the
horofunction representing $\xi$ that vanishes at $p$. Then there exists a sequence 
$(x_n)\subset\OM$, $x_n\to x$, such that 
\[
h_{p,\,\xi}(\cdot)=\lim_{n\to\infty}\big(\koba_\OM(\cdot, x_n)-\koba_\OM(p, x_n)\big). 
\]
Now suppose $z\in H_p(\xi, R)$ then 
\[
\liminf_{w\to x}\big(\koba_\OM(z, w)-\koba_\OM(p, w)\big)\leq\lim_{n\to\infty}(\koba_\OM(z, x_n)-\koba_\OM(p, x_n))=h_{p,\,\xi}(z)<\tfrac{1}{2}\log R.
\]
This implies that for each $\xi\in\mathscr{H}(x)$, $H_p(\xi, R)\subset H^b_p(x, R)$. Therefore, we have 
\[
\bigcup_{\xi\in\mathscr{H}(x)}H_p(\xi, R)\subset H^b_p(x, R). 
\]
Now let $z\in H^b_p(x, R)$,
then there exists a sequence $(x_n)$ such that $x_n\to x$ and
\[
\liminf_{w\to x}\big(\koba_\OM(z, w)-\koba_\OM(p, w)\big)=
\lim_{n\to\infty}(\koba_\OM(z, x_n)-\koba_\OM(p, x_n))<\tfrac{1}{2}\log R. 
\]
Consider the sequence $(\koba_\OM(\cdot, x_n)-\koba_\OM(p, x_n))$, as $x_n\to x$, by the 
compactness of $\clos{\OM}^{\koba}$, a subsequence of this sequence converges to a
horofunction that vanishes at $p$. If we write this function as $h_{p,\,\xi}$ then
$[\,h_{p,\,\xi}\,]=\xi\in\mathscr{H}(x)$. The inequality above implies that 
\[
h_{p,\,\xi}(z)<\tfrac{1}{2}\log R
\]
whence $z\in H_p(\xi, R)$. Since $z$ is arbitrary, we are done. 
\smallskip

For small horospheres, choose $z\in H^s_p(x, R)$, where, recall
\begin{equation*}\label{E:smallhoro}
H^s_p(x, R)=\left\{z\in\OM\,:\,\limsup_{w\to x}\big(\koba_\OM(z, w)-\koba_\OM(p, w)\big)<
\tfrac{1}{2}\log R\right\},
\end{equation*}
and let $\xi\in\mathscr{H}(x)$ be given. 
Consider $h_{p,\,\xi}$ as above and let
$(x_n)\subset\OM$, $x_n\to x$, such that 
$h_{p,\,\xi}(\cdot)=\lim_{n\to\infty}\big(\koba_\OM(\cdot, x_n)-\koba_\OM(p, x_n)\big)$. Therefore, 
\begin{align*}
h_{p,\,\xi}(z)&=\lim_{n\to\infty}\big(\koba_\OM(z, x_n)-\koba_\OM(p, x_n)\big)
\leq \limsup_{w\to x}\big(\koba_\OM(z, w)-\koba_\OM(p, w)\big)<\tfrac{1}{2}\log R.
\end{align*}
Since $\xi$ is arbitrary, we get that 
\[
H^s_p(x, R)\subset \bigcap_{\xi\in\mathscr{H}(x)}H_p(\xi, R).
\]
Now let $z\in\OM$ be such that $z\in H_p(\xi, R)$ for 
{\em every} $\xi\in\mathscr{H}(x)$. Now choose a sequence $(x_n)$ such that 
$x_n\to x$ and 
\[
\limsup_{w\to x}\big(\koba_\OM(z, w)-\koba_\OM(p, w)\big)=
\lim_{n\to\infty}(\koba_\OM(z, x_n)-\koba_\OM(p, x_n)).
\]
Again, by the compactness of $\clos{\OM}^{\koba}$, a subsequence of
$(\koba_\OM(\cdot, x_n)-\koba_\OM(p, x_n))_{n\geq 1}$
converges to a
horofunction that vanishes at $p$. If we write this function as $h_{p,\,\xi}$ then
$[\,h_{p,\,\xi}\,]=\xi\in\mathscr{H}(x)$. By our assumption 
$h_{p,\,\xi}(z)<(1/2)\log R$, and the equality above implies that 
$z\in H^s_p(x, R)$. 
\end{proof}

For bounded convex domains, the horofunctions are quasi-convex. 
\begin{proposition}\label{P:horoconvex}
Let $\OM\Subset\C^n$ be a convex domain. Given $\xi\in\partial\clos{\OM}^{\koba}$,
let $x\in\bdy\OM$ be such that $\xi\in\mathscr{H}(x)$. Then, for a fixed 
$p\in\OM$, we have
\begin{equation}\label{E:conhromax}
h_{p,\,\xi}\big(tz_1+(1-t)z_2\big)\leq\max\left\{h_{p,\,\xi}(z_1), h_{p,\,\xi}(z_2)\right\},
\quad\forall z_1, z_2\in\OM, \ t\in [0, 1].
\end{equation}
In particular, each horoball $H_p(\xi, R)$ is convex, and moreover, 
$x\in\overline{H_p(\xi, R)}$. 
\end{proposition}
\begin{proof}
By the definition of \(h_{p,\,\xi}\), since $\xi\in\mathscr{H}(x)$,
there exists a sequence \((x_n)\) such that $x_n\to x$ and 
$h_{p,\,\xi}(\cdot)=
\lim_{n\to\infty}\big(\koba_\OM(\cdot, x_n)-\koba_\OM(p, x_n)\big)$.
Given $z_1, z_2\in\OM$, $t\in [0, 1]$, and $x_n$,
the convexity of $\OM$ implies that
$\koba_\OM(tz_1+(1-t)z_2, x_n)\leq\max\left\{\koba_\OM(z_1, x_n), \koba_\OM(z_2, x_n)\right\}$, 
see e.g., \cite[Proposition~2.3.46]{Abate:iteration89}. 
Therefore,
\[
\koba_\OM(tz_1+(1-t)z_2, x_n)-\koba_\OM(p, x_n)\leq
\max\left\{\koba_\OM(z_1, x_n)-\koba_\OM(p, x_n),\, \koba_\OM(z_2, x_n)-\koba_\OM(p, x_n)\right\}.
\]
Taking limit as $n\to\infty$ in the above equation, \eqref{E:conhromax} follows. Clearly this implies 
that $H_p(\xi, R)$ is convex. 
To show that $x\in\overline{H_p(\xi, R)}$, it is enough to show that for every
$z\in H_p(\xi, R)$, $s\in(0, 1)$, $sz+(1-s)x\in H_p(\xi, R)$. Write 
\begin{align*}
    \koba_\OM(sz+(1-s)x, x_n)-\koba_\OM(p, x_n)&
    \leq \big(\koba_\OM(sz+(1-s)x_n, x_n)-\koba_\OM(p, x_n)\big)\\
    & +\big(\koba_\OM(sz+(1-s)x, x_n)-\koba_\OM(sz+(1-s)x_n, x_n)\big)
    \end{align*}
Note that again the convexity of $\OM$ implies that \(\koba_\OM(sz+(1-s)x_n, x_n)\leq \koba_\OM(z, x_n)\).
Also,
\[
\lim_{n\to\infty}\koba_\OM(sz+(1-s)x, x_n)-\koba_\OM(sz+(1-s)x_n, x_n)\leq 
\lim_{n\to\infty}\koba_\OM(sz+(1-s)x, sz+(1-s)x_n)=0.
\]
Therefore, we have 
\[
\lim_{n\to\infty}\koba_\OM(sz+(1-s)x, x_n)-\koba_\OM(p, x_n)
    \leq \lim_{n\to\infty}\koba_\OM(z, x_n)-\koba_\OM(p, x_n)<\tfrac{1}{2}\log R.
\]
Since $s\in(0,1)$ is arbitrary, it follows that $x\in\overline{H_p(\xi, R)}$. 
\end{proof}
The following result\,---\,already established by Abate--Raissy \cite{MJ:2014}\,---\,is a direct 
consequence of Lemma~\ref{lm:smallbighoro} and Proposition~\ref{P:horoconvex}. 

\begin{corollary}\label{Cor:smallconbigstar}
Let $\OM$ be a bounded convex domain. Then for any $p\in\bdy\OM$, $x\in\bdy\OM$ and 
$R>0$, \(H^s_p(x, R)\) is convex; \(\overline{H^b_p(x, R)}\) is star-shaped with respect to $x$.
\end{corollary}
\smallskip

\subsection{Limit set and target set}\label{SS:wolf}
In this subsection, we recall the notion of the limit set and target set 
for a holomorphic self-map. We also prove two auxiliary lemmas that
will help us in the proof of our main theorems in the next subsection. 

\begin{definition}\label{Def:limset}
Let \(\mfd{X}\) be a complex manifold and let 
\( \Omega \subsetneq\mfd{X} \) be a subdomain. 
Consider the space of continuous functions from $\OM$ to $\clos{\OM}^{\rm End}$,
together with compact-open topology.
Let $f : \Omega \lrarw \Omega$
be a holomorphic self-map, or more generally, a self-map
that is non-expansive with respect to the Kobayashi distance on \( \Omega \).
A map \( g : \Omega \lrarw \clos{\OM}^{\rm End} \) is called a \emph{limit map}
of the sequence of iterates \( (f^k) \)
if there exists a subsequence \( \{f^{k_j}\} \) such that \( f^{k_j} \to g \) 
in the compact-open topology. We denote the set of all such limit maps by \( \Gamma(f) \).
The \emph{target set} of \( f \), denoted by \( T(f) \),
is defined as the union of the images of all maps in \( \Gamma(f) \):
\[
T(f) := \bigcup_{g\,\in\,\Gamma(f)} g(\Omega).
\]
\end{definition}
\noindent Our first lemma will be useful in showing that the limit set 
of a holomorphic self-map\,---\,whose iterates are
compactly divergent\,---\,is nonempty. 

\begin{lemma}\label{L:comopen}
Let \(\OM\subsetneq\mfd{X}\) be a Kobayashi hyperbolic domain in a 
complex manifold \(\mfd{X}\) that satisfies
\begin{equation}\label{E:BSP}
\liminf_{(z,\,w)\to(x,\,y)}\koba_\OM(z,\,w)>0 \ \ \ 
\text{for every $x\neq y\in\bdy\clos{\OM}^{\rm End}$}. 
\end{equation}
Let $f:\OM\lrarw\OM$ be a holomorphic map, and 
\(\xi\in\bdy\clos{\OM}^{\rm End}\) be such that for every $z\in\OM$, 
$f^n(z)\to\xi$. Then $f^n\to\xi$ in the compact-open topology. 
\end{lemma}
\begin{proof}
It is sufficient to show that given $p\in\OM$, $r>0$ such that 
$B_{\koba}(p, r)\subset\subset\OM$, and a neighborhood $U_\xi$ of $\xi$, there exists
$N_0\in\nat$ such that
\[
f^n(B_{\koba}(p, r))\subset U_\xi \ \ \ \text{for all $n\geq N_0$}.
\]

\noindent{\bf Case 1.} $\xi\in\bdy\clos{\OM}^{\rm End}$ is an end. 
\smallskip

\noindent Fix an exhaustion $(K_j)_{j\geq 1}$ of $\clos{\OM}$ by compact sets as described
in Subsection~\ref{SS:endcomp}.
Suppose $\xi=(F_j)_{j\geq 1}$ where each $F_j$ is a connected component of 
$\clos{\OM}\setminus K_j$. Then we can find a $j_0\in\nat$ such that 
$\widehat{F}_{j_0}\subset U_{\xi}$ where $(\widehat{F}_{j})_{j\geq 1}$
is the neighborhood basis of $\xi$ as described in Subsection~\ref{SS:endcomp}. 
Clearly, there exists $N_1(p, j_0)\in\nat$
such that 
\begin{equation}\label{E:ittcompl}
f^n(p)\in F_{j_0}\subset\widehat{F}_{j_0} \ \ \ 
\text{for all $n\geq N_1(p, j_0)$}.
\end{equation}

\noindent{\bf Claim.} There exists $N_2(p, j_0)\in\nat$ such that 
\begin{equation}\label{E:ittballcompleK}
f^n(B_{\koba}(p, r))\cap K_{j_0}=\emptyset \ \ \ 
\text{for all $n\geq N_2(p, j_0)$}.
\end{equation}

\noindent Suppose that the above is not true, then there exist $(n_j)_{j\geq 1}$ 
and $(x_j)_{j\geq 1}\subset B_{\koba}(p, r)$ such that $f^{n_j}(x_j)\in K_{j_0}$. 
Extract a subsequence $(x_{j_\nu})$ such that 
$x_{j_\nu}\to x_0\in\clos{B_{\koba}(p, r)}$. Then
\begin{equation}\label{E:comopen_1}
\koba_\OM(f^{n_{j_\nu}}(x_{j_{\nu}}),  f^{n_{j_\nu}}(x_0))\leq
\koba_\OM(x_{j_{\nu}}, x_0)\to 0, \ \ \ \text{as $\nu\to\infty$}. 
\end{equation}
The limit points of the sequence $(f^{n_{j_\nu}}(x_{j_{\nu}}))_{\nu\geq 1}$
are in $\clos{\OM}\cap K_{j_0}$. Since $f^{n_{j_\nu}}(x_0)\to\xi$,
\eqref{E:comopen_1} implies they are in $\bdy\OM$. 
But that contradicts our hypothesis \eqref{E:BSP}. Hence the claim follows.
\smallskip

Let $N_0=\max\{N_1(p, j_0), N_2(p, j_0)\}$ then for any $n\geq N_0$, both \eqref{E:ittcompl} and 
\eqref{E:ittballcompleK} are satisfied. Note that, for each $n\geq N_0$, 
$f^n(B_{\koba}(p, r))$ is a connected set, disjoint from $K_{j_0}$ and intersecting $F_{j_0}$. This implies 
that
\[
f^n(B_{\koba}(p, r))\subset F_{j_0}\subset U_\xi \ \ \ \text{for all $n\geq N_0$}.
\]
Hence we are done in this case. 
\smallskip

\noindent{\bf Case 2.} $\xi\in\bdy\OM$. 
\smallskip

\noindent Choose a hermitian metric $h$ on $\mfd{X}$ and let $R>0$ be small enough such that  
$B_{d_h}(\xi, R)\cap\clos{\OM}\subset U_\xi$. Here, $B_{d_h}(\xi, R)$ denotes the open ball
in $\mfd{X}$, with respect
to the distance induced by $h$, with center $\xi$ and radius $R>0$. 
We claim that there exists $N_0\in\nat$ such that 
\[
f^n(B_{\koba}(p, r))\subset B_{d_h}(\xi, R) \ \ \ \text{for all $n\geq N_0$}. 
\]
Suppose this is not true, then there exist $(n_j)_{j\geq 1}$, and $(x_j)\subset B_{\koba}(p, r)$ such that 
$f^{n_j}(x_j)\notin B_{d_h}(\xi, R)$ for all $j$. Passing to a subsequence, we can assume that
$x_j\to x_0\in\clos{B_{\koba}(p, r)}$. Therefore, 
\[
\koba_\OM(f^{n_j}(x_j), f^{n_j}(x_0))\leq \koba_\OM(x_j, x_0)\to 0 \ \ \ \text{as $j\to\infty$}.
\]
Since $f^{n_j}(x_0)\to\xi\in\bdy\OM$, the limit points of the sequence
$f^{n_j}(x_j)$ must belong to 
$\bdy\clos{\OM}^{\rm End}\setminus B_{d_h}(\xi, R/2)$.
But this will contradict
\eqref{E:BSP} once again. Thus, our claim follows and we are done in this case too. 
\end{proof}
The condition \eqref{E:BSP} in the above lemma is equivalent to 
$\OM$ being {\em hyperbolically embedded} in $\mfd{X}$. All domains that 
satisfy any type of visibility property satisfy this condition; see e.g. 
\cite[Lemma~3.7]{CGMS:2024}. Also, this condition 
is referred to as the boundary separation property in the latter article. 
We also need the following lemma. 

\begin{lemma}\label{E:endboundcon}
    Let $\OM\subsetneq\mfd{X}$ be a complete Kobayashi hyperbolic domain of a complex
    manifold $\mfd{X}$. Suppose that
    \begin{equation}\label{E:kobdivc}
     \lim_{(z, w)\to (x, y)}\koba_\OM(z, w)=\infty \ \ \ 
      \text{for any $x\neq y\in\bdy\OM$}.
    \end{equation}
    Then the above also holds for any two distinct $p, q\in\bdy\clos{\OM}^{\rm End}$. 
\end{lemma}
\begin{proof}
    Given $p\neq q\in\bdy\clos{\OM}^{\rm End}$, if both are in $\bdy\OM$, we 
    are trivially done. So assume that at least one of $p, q$ is an end.
    Now suppose \eqref{E:kobdivc} does not hold for $p, q$. Then there exist sequences 
    $(z_n)$, $(w_n)$, $z_n\to p$ and $w_n\to q$ such that
    \[
    \sup\koba_\OM(z_n, w_n)=M<\infty.
    \]
    Let $\gamma_n:[a_n, b_n]\lrarw\OM$ be a Kobayashi geodesic joining $z_n$ and 
    $w_n$. Then we claim that $(\gamma_n)$ eventually avoids every compact
    $K\subset\OM$. Otherwise, there exists a compact
    $K\subset\OM$ such that\,---\,without loss of any generality\,---\,we 
    may assume that each $\gamma_n$ intersects $K$. Let $t_n\in[a_n, b_n]$ be
    such that $\gamma_n(t_n)\in K$. Since \(z_n, \gamma_n(t_n), w_n\in{\rm range}(\gamma_n)\)
    for each \(n\in\nat\), we have
   \[
   \koba_\OM(z_n, K)\leq\koba_\OM(z_n,\gamma_n(t_n))\leq \koba_\OM(z_n, w_n)<M, 
   \]
   which implies that $(z_n)$ lies in a bounded Kobayashi neighbourhood of the 
   compact set $K$. The same holds for $(w_n)$. Since $\OM$ is complete 
   hyperbolic, this implies that both $(z_n)$, $(w_n)$ lie in a compact subset of 
   $\OM$, which contradicts that $p, q\in\bdy\clos{\OM}^{\rm End}$.
   \smallskip

   Using the fact that $(\gamma_n)$ eventually avoids
   every compact subset of $\OM$, one can find points 
   $x\neq y\in\bdy\OM$ and sequences $(s_n)_{n\geq 1}, (t_n)_{n\geq 1}\subset [a_n, b_n]$ such that
   $\gamma_n(s_n)\to x$ and $\gamma_n(t_n)\to y$. We refer the reader to 
   Lemma~2.2 and the proof of Theorem~1.6 in \cite{CGMS:2024}
   for the precise argument. Observe that
   \[
   \liminf_{(z, w)\to(x, y)}\koba_\OM(z,w)
   \leq\liminf_{n\to\infty}\koba_\OM(\gamma(s_n), \gamma(t_n))
   \leq \liminf_{n\to\infty}\koba_\OM(z_n, w_n)<\infty, 
   \]
   which contradicts our hypothesis, whence the result follows. 
\end{proof}

\subsection{The proofs of Theorem~\ref{T:wolfftheorem}, Theorem~\ref{T:DW} and
Propostion~\ref{prop:bconvex}.}
We now have all the ingredients to present:
    \begin{proof}[The proof of Theorem~\ref{T:wolfftheorem}]
    Fix a point $p\in\OM$. Since $(f^n)_{n\geq 1}$ is compactly divergent,
    the completeness of $\koba_\OM$ implies that
    \[
    \lim_{n\to\infty}{\koba_\OM\big(p,f^n(p)\big)}=\infty.
    \]
    There exists a strictly increasing sequence of natural numbers \((n_\nu)_{\nu\geq 1}\) such that 
    \begin{equation}\label{E:fundain}
    \koba_\OM(p,f^{n_\nu+1}(p))>\koba_\OM(p,f^{n_\nu}(p)) \ \ \ \forall \nu\in\nat.
    \end{equation}
    Let us write \(w_{n_\nu}=f^{n_\nu}(p)\). Since \(\clos{\OM}^{\rm End}\) is 
    sequentially compact, by passing to subsequences if necessary,
    we can assume that $w_{n_\nu}\to x\in\bdy\clos{\OM}^{\rm End}$ in the topology of end-point
    compactification, and
    \(i_{\koba}(w_{n_\nu})\to\xi\in\bdy\clos{\OM}^{\koba}\) in the topology of metric compactification, as $\nu\to\infty$. 
    Clearly, by definition, \(\xi\in\mathscr{H}(x)\).
    \smallskip

    \noindent\textbf{Claim 1:} \(f(w_{n_\nu})\to x\) as \(\nu\to\infty\).
    \smallskip

    \noindent Suppose, to get a contradiction, that the above is not true. Then, 
    up to passing to a subsequence, we can assume that
    $f(w_{n_\nu})\to y$ as $\nu\to\infty$ for some $x\neq y\in\bdy\clos{\OM}^{\rm End}$. Observe that
    \[
    \koba_\OM\big(w_{n_\nu},f(w_{n_\nu})\big)=\koba_\OM\big(f^{n_\nu}(p),f^{n_\nu+1}(p)\big)
    \leq\koba_\OM\big(p,f(p)\big)<\infty.
    \]
    Hence, $\liminf_{(z,w)\to (x,y)}{\koba_\OM(z,w)<\infty}$. However, Lemma~\ref{E:endboundcon}
    implies that this is not possible, whence the claim.
    \smallskip

    Once again, passing to a subsequence if necessary, we can further assume that 
    \(i_{\koba}(f(w_{n_\nu}))\to\xi_1\) in the topology of metric compactification, where 
    \(\xi_1\in\mathscr{H}(x)\). 
    \smallskip

    \noindent{\bf Claim 2:} For any $R>0$, we have 
    \begin{equation}\label{E:horointohoro}
    f(H_p(\xi, R)\subset H_p(\xi_1, R).
    \end{equation}
    Recall, \(H_p(\xi, R)\), \(H_p(\xi_1, R)\) are the 
    horoballs centered at \(\xi\), \(\xi_1\) respectively, with respect to the pole $p$ and radius $R>0$.
    To see the above, fix an \(R>0\) and let \(z\in H_p(\xi,R)\). Then 
    \begin{align}\label{E:imhoro}
        \koba_\OM\big(f(z),f(w_{n_\nu})\big)-\koba_\OM\big(p,f(w_{n_\nu})\big)
        &\leq \koba_\OM(z,w_{n_\nu})-\koba_\OM(p,f(w_{n_\nu}))\nonumber\\
        &=\koba_\OM(z,w_{n_\nu})-\koba_\OM(p,w_{n_\nu})+\koba_\OM(p,w_{n_\nu})-\koba_\OM(p,f(w_{n_\nu}))\nonumber\\
        &\leq\koba_\OM(z,w_{n_\nu})-\koba_\OM(p,w_{n_\nu}),
    \end{align}
    where the last inequality above follows from \eqref{E:fundain}.
    It follows from \eqref{E:imhoro} that 
    \[
    h_{p,\,\xi_1}(f(z))\leq h_{p,\,\xi}(z)<\tfrac{1}{2}\log R,
    \]
    where \(h_{p,\,\xi_1}\) and \(h_{p,\,\xi}\) are the unique horofunctions representing 
    $\xi_1, \xi$ respectively, and that vanish at the point \(p\). From this our claim follows. 
    \smallskip

    Applying the method of the proof of Claim~1 and Claim~2 above, inductively, 
    we see that for each $n\in\nat$, there exists $\xi_n\in\mathscr{H}(x)$ such that
    \[
    f^n(H_p(\xi, R))\subset H_p(\xi_n, R) \ \ \ \text{for all \(R>0\)}. 
    \]
    Since $H^b_p(x, R)$ is the union of all the
    horoballs \(H_p(\eta, R)\), \(\eta\in\mathscr{H}(x)\), 
    the above inclusion implies \eqref{E:imhoroballbig}.
    Note that $H^s_p(x, R)$ is the intersection of all horoballs 
    \(H_p(\eta, R)\), \(\eta\in\mathscr{H}(x)\), 
    and hence is contained in \(H_p(\xi, R)\), from which the inclusion \eqref{E:imsmallbig}
    follows.
\end{proof}

As mentioned in Section~1, after the statement of Theorem~\ref{T:wolfftheorem},
the property that iterates map a certain horoball into
the big horosphere associated to it, helps us in proving the following 
result about the location of the target set. 

\begin{lemma}\label{L:limitset}
Let $\OM\subsetneq\mfd{X}$ be a complete Kobayashi hyperbolic domain
of a complex manifold $\mfd{X}$ satisfying property \eqref{E:kobdivc}.
Let $f:\OM\lrarw\OM$ be a holomorphic map such that 
\((f^n)_{n\geq 1}\) is compactly divergent. Then \(\Gamma(f)\) is nonempty and 
each \(h\in\Gamma(f)\) is a constant function. Moreover, 
\[
T(f)\subset\bigcap_{R>0}\,\overline{H^b_p(x, R)},
\]
where \(p\in\OM\) is a fixed point, and $x\in\bdy\clos{\OM}^{\rm End}$
is as given by Theorem~\ref{T:wolfftheorem}.
\end{lemma}
\begin{proof}
Since $(f^n)_{n\geq 1}$ is compactly divergent and $(\OM, \koba_\OM)$ is complete,
given $p\in\OM$, we have 
\[
\lim_{n\to\infty}\koba_\OM(p, f^n(p))=\infty.
\]
As \(\clos{\OM}^{\rm End}\) is sequentially compact, there exist a 
a subsequence \((f^{n_j})_{j\geq 1}\) and a point $y\in\bdy\clos{\OM}^{\rm End}$
such that $f^{n_j}(p)\to y$.
We claim that $f^{n_j}(z)\to y$ for every $z\in\OM$. 
Consider a limit point $w\in\bdy\clos{\OM}^{\rm End}$ of the sequence 
$(f^{n_j}(z))$. Without loss of generality, we may assume that 
$f^{n_j}(z)\to w$. Since 
\[
\koba_\OM(f^{n_j}(z), f^{n_j}(p))\leq \koba_\OM(z, p)<\infty,
\]
it follows from our hypothesis that $w=y$. This implies that the only 
limit point of $(f^{n_j}(z))$ is $y$.
Now an application of Lemma~\ref{L:comopen} implies that $f^{n_j}\to y$ in the compact-open topology,
and hence, $y\in\Gamma(f)$. Here, $y$ also denotes the constant map.  
\smallskip

Let $g\in\Gamma(f)$, clearly $g(\OM)\subset\bdy\clos{\OM}^{\rm End}$.
Suppose $g$ is not constant, and let $y_1\neq y_2\in\bdy\clos{\OM}^{\rm End}$ be such that 
$y_1=g(z)$, $y_2=g(w)$. Let \( (f^{k_j}) \) be such that \( f^{k_j} \to g \) 
in the compact-open topology. Then
 \[
  \koba_\OM(f^{k_j}(z), f^{k_j}(w))\leq\koba_\OM(z, w) \ \ \ \forall j\in\nat. 
  \]
  Since \(f^{k_j}(z)\to y_1\) and \(f^{k_j}(w)\to y_2\), the above inequality contradicts
  our hypothesis. Therefore, $g$ is constant.
  \smallskip
  
  Let $x\in\bdy\clos{\OM}^{\rm End}$
  and $\xi\in\mathscr{H}(x)$ be as given by Theorem~\ref{T:wolfftheorem}. Then for any $R>0$, 
  and any $n\in\nat$, we have 
  \[
  f^{n}(H_p(\xi, R))\subset H^b_p(x, R) \ \ \ \forall R>0.
  \]
  Observe that given $g\in\Gamma(f)$, since $g$ is a constant function, 
  $g(\OM)=g(H_p(\xi, R))$ for any $R>0$.  
  Therefore, for each $R>0$ we have 
  \[
  g(\OM)=g(H_p(\xi, R))=\lim_{j\to\infty}f^{k_j}(H_p(\xi, R))\subset\overline{H^b_p(x, R)}\cap\bdy\clos{\OM}^{\rm End}. 
  \]
  This implies 
  \[
  g(\OM)\subset\bigcap_{R>0}\big(\overline{H^b_p(x, R)}\cap\bdy\clos{\OM}^{\rm End}\big),
  \]
  and since $g\in\Gamma(f)$ is arbitrary, the result follows. 
  \end{proof}
  We are now ready to present the proof of Theorem~\ref{T:DW}. 

\begin{proof}[The proof of Theorem~\ref{T:DW}]
    Since $\OM$ is complete Kobayashi hyperbolic domain, $\OM$ is taut.
    Therefore, it follows from \cite[Theorem~2.4.3]{Abate:iteration89} that
    the sequence of iterates $(f^n)_{n\geq 1}$ is either relatively
    compact or compactly divergent. 
    Clearly, part $(a)$ follows when $(f^n)_{n\geq 1}$ is relatively compact.
    Now suppose $(f^n)_{n\geq 1}$ is compactly divergent, 
    and let $p\in\OM$ be given. Let $\Gamma(f)$ be the limit set of $f$.
    Then, by Lemma~\ref{L:limitset}, \(\Gamma(f)\) is nonempty, and 
    each $g\in\Gamma(f)$ is a constant $x_g\in \bdy\clos{\OM}^{\rm End}$ such that 
    \[
    x_g\in \bigcap_{R>0}\overline{H^b_p(x, R)} 
    \]
   for some $x\in\bdy\clos{\OM}^{\rm End}$. 
   By our assumption, there exists a $y\in\bdy\clos{\OM}^{\rm End}$ such that 
   \[
   \bigcap_{R>0}\overline{H^b_p(x, R)}=\{y\},
   \]
therefore, $g\equiv y$, where $g\in\Gamma(f)$ is arbitrary. From this the result follows. 
\end{proof}
We now present the proof of Proposition~\ref{prop:bconvex} but first we recall 
the following result. 

\begin{result}[{paraphrasing \cite[Lemma~8]{MJ:2014}}]\label{Res:invhoro}
Let $D\Subset\C^d$ be a bounded convex domain. Let $f:D\lrarw D$ be a holomorphic map
without a fixed point. Then there exist $x\in\bdy D$ and $\xi\in\mathscr{H}(x)$ such that,
for a fixed point $p\in D$, we have 
\[
f(H_p(\xi, R))\subset H_p(\xi, R)\quad \forall R>0.
\]
\end{result}
\noindent Abate--Raissy stated the above result in terms of a 
{\em horosphere sequence} $(x_\nu)$ converging to $x$. In our language, this basically 
means $i_\koba(x_\nu)\to\xi$ for some $\xi\in\mathscr{H}(x)$. 
We are now ready to present: 

\begin{proof}[The proof of Proposition~\ref{prop:bconvex}]
Suppose $f$ does not have any fixed point. Then by
\cite[Theorem~2.4.20]{Abate:iteration89} the sequence $(f^n)$ is 
compactly divergent. Appealing to Montel's Theorem, as $\OM$ is bounded, 
we see that $\Gamma(f)$ is nonempty. Moreover, the boundary 
divergence condition~\ref{E:kobdivc}
implies\,---\,as shown in the proof of Lemma~\ref{L:limitset}\,---\,that each 
function in $\Gamma(f)$ is constant. Now regarding the target set, we first observe 
that by Result~\ref{Res:invhoro}, there exist $x\in\bdy D$ and
$\xi\in\mathscr{H}(x)$ such that,
for a fixed point $p\in D$, and for any $n\in\nat$, we have 
\[
f^n(H_p(\xi, R))\subset H_p(\xi, R)\quad \forall R>0.
\]
Exactly proceeding as in the proof of 
Lemma~\ref{L:limitset}, we see that the target set 
\[
T(f)\subset\bigcap_{R>0}\clos{H_p(\xi,\,R)}. 
\]
By our assumption, the set in the right-side of the above inclusion is $\{x\}$. 
(Observe that Proposition~\ref{P:horoconvex} implies that this set contains $x$). 
This completes the proof. 
\end{proof}

\section{Continuous extension of the canonical map}\label{S:context}
In this section, we present the proof of Theorem~\ref{thm:mettoend}. We also 
establish a necessary condition and a sufficient condition for the 
existence of the continuous map  $F:\clos{\OM}^{\koba}\lrarw\clos{\OM}^{\rm End}$ such that 
\begin{equation}\label{E:extcanodup}
F\circ i_{\koba}={i}_\OM \ \ \ \text{on $\OM$}. 
\end{equation}
In particular, we observe
that the property ensuring such an extension appears to be weaker 
than the geodesic visibility property. 
\smallskip

\subsection{The proof of Theorem~\ref{thm:mettoend}}\label{SS:proofmettoend}
\begin{proof}
    Suppose there exists a continuous surjective map
    $F:\clos{\OM}^{\koba}\lrarw\clos{\OM}^{\rm End}$ that
    satisfies \eqref{E:extcanodup}.
    \smallskip

    \noindent{\bf Claim.} For any $x\in\bdy\clos{\OM}^{\rm End}$, 
    $F^{-1}(x)=\mathscr{H}(x)$. 
    \smallskip

    \noindent Let $\xi\in F^{-1}(x)$ and let $(x_n)\subset\OM$ be such that
    $i_{\koba}(x_n)\to\xi$. This implies that $F\circ i_{\koba}(x_n)\to F(\xi)$. Since 
    $F$ satisfies \eqref{E:extcanodup}, and $F(\xi)=x$, we get that
    $x_n\to x$. Since $(x_n)$ has no limit point in $\OM$,
    it follows\,---\,by Remark~\ref{rm:charmetb}\,---\,that
    $\xi\in\bdy\clos{\OM}^{\koba}$ whence, by definition,
    $\xi\in\mathscr{H}(x)$ and $F^{-1}(x)\subset\mathscr{H}(x)$.
    Conversely, suppose $\xi\in\mathscr{H}(x)$ then there exists a sequence 
    $(x_n)\subset\OM$ such that $x_n\to x$ and $i_{\koba}(x_n)\to\xi$. Therefore, 
    $F\circ i_{\koba}(x_n)\to F(\xi)$ but $F\circ i_{\koba}(x_n)=x_n$, and since
    $x_n\to x$, we get that $F(\xi)=x$. So $F(\mathscr{H}(x))\subset\{x\}$
    for all $x\in\bdy\clos{\OM}^{\rm End}$. 
    This establishes the claim from which the only if part of the statement of
    the theorem follows.
    \smallskip

    Suppose now $\mathscr{H}(x)\cap\mathscr{H}(y)=\emptyset$ for all
    $x\neq y\in\bdy\clos{\OM}^{\rm End}$. Given $\xi\in\bdy\clos{\OM}^{\koba}$, there exists a sequence 
    $(x_n)_{n\geq 1}$ such that $i_{\koba}(x_n)\to\xi$ in the topology of $\clos{\OM}^{\koba}$.
    Note that\,---\,by Remark~\ref{rm:charmetb}\,---\,$(x_n)_{n\geq 1}$ does
    not have a limit point in $\OM$. Since $\clos{\OM}^{\rm End}$ is sequentially 
    compact, there exists a sequence $(x_{n_k})_{k\geq 1}$ such that $x_{n_k}\to x\in\bdy\clos{\OM}^{\rm End}$, as $k\to\infty$, 
    in the topology of $\clos{\OM}^{\rm End}$. This implies, by definition, $\xi\in\mathscr{H}(x)$, i.e.,
    we have 
    \begin{equation}\label{E:metbfib}
     \bdy\clos{\OM}^{\koba}=\clos{\OM}^{\koba}\setminus i_{\koba}(\OM)=\bigsqcup_{x\in\bdy\clos{\OM}^{\rm End}}\mathscr{H}(x). 
    \end{equation}
    Consider the map
    $F:\clos{\OM}^{\koba}\lrarw\clos{\OM}^{\rm End}$ defined by 
    \[
    F(\xi):=
    \begin{cases}
        i_{\koba}^{-1}(\xi), & \ \ \ \xi\in i_{\koba}(\OM),\\ 
        x, & \ \ \ \text{if $\xi\in\mathscr{H}(x)$.}
    \end{cases}
    \]
    Note $F$ is a well-defined, surjective map that satisfies condition \eqref{E:extcanodup}.
    We now show that $F$ is continuous. It is sufficient to show that
    $F$ is continuous on $\bdy\clos{\OM}^{\koba}$.
    Let $(\xi_n)_{n\geq 1}\subset\clos{\OM}^{\koba}$ be such that $\xi_n\to\xi\in\bdy\clos{\OM}^{\koba}$,
    in the topology of $\clos{\OM}^{\koba}$. 
    %We have to show that $F(\xi_n)\to F(\xi)$, in the topology of $\clos{\OM}^{\rm End}$.
    We may assume, without loss of any generality, that either 
    $(\xi_n)_{n\geq 1}\subset i_{\koba}(\OM)$ or $(\xi_n)_{n\geq 1}\subset\bdy\clos{\OM}^{\koba}$. 
    \smallskip

    \noindent{\bf Case 1.} $(\xi_n)_{n\geq 1}\subset i_{\koba}(\OM)$.
    \smallskip

    \noindent In this case, let $(x_n)_{n\geq 1}\subset\OM$ be such that 
    $x_n=F(\xi_n)=i_{\koba}^{-1}(\xi_n)$ for all $n$. Also, let
    $F(\xi)=x\in\bdy\clos{\OM}^{\rm End}$, i.e., there exists a sequence
    $(z_n)_{n\geq 1}\subset\OM$ such that $z_n\to x$ and $i_{\koba}(z_n)\to \xi$.
    Let $y$ be a limit point of $(x_n)_{n\geq 1}$, since $i_{\koba}(x_n)=\xi_n\to\xi\in\bdy\clos{\OM}^{\koba}$, 
    $y\in\bdy\clos{\OM}^{\rm End}$. Let $(x_{n_k})_{k\geq 1}$ be such that
    $x_{n_k}\to y\in\bdy\clos{\OM}^{\rm End}$. Note $i_{\koba}(x_{n_k})\to\xi$ and hence by definition
    $\xi\in\mathscr{H}(y)\cap\mathscr{H}(x)$. By our assumption, we have $y=x$.
    Since $y$ is arbitrary, $(x_n)_{n\geq 1}$ converges to $x$ and therefore $F(\xi_n)$ converges to $F(\xi)$, and hence we are done. 
    \smallskip

    \noindent{\bf Case 2.} $(\xi_n)_{n\geq 1}\subset\bdy\clos{\OM}^{\koba}$.
    \smallskip

    \noindent Let $F(\xi_n)=e_n\in\bdy\clos{\OM}^{\rm End}$. 
    To show that $F(\xi_n)$ converges 
    to $F(\xi)$, we will show that if $e\in\bdy\clos{\OM}^{\rm End}$ is a 
    limit point of $(e_n)$, then $e=F(\xi)$.
    Without loss of any
    generality, we may assume that $e_n\to e\in\bdy\clos{\OM}^{\rm End}$.
    Note that, by the definition of $F$, for each $n$, there 
    exists a sequence $(z^n_\nu)_{\nu\geq 1}$ such that $z^n_\nu\to e_n$
    and $i_{\koba}(z^n_\nu)\to \xi_n$ as $\nu\to\infty$.
    %To show that
    %$F(\xi)=e$, we have to show that there exists a sequence $(x_\nu)\subset\OM$ 
    %such that $x_\nu\to e$ and $i_{\koba}(x_\nu)\to\xi$ as $\nu\to\infty$. 
    \smallskip

    \noindent{\bf Subcase (a).} $e\in\bdy\OM$, i.e., $e$ is not an end. 
    \smallskip 

    \noindent Since $e_n\to e$ then except finitely many $n$, 
    all $e_n\in\bdy\OM$.
    Without loss of any generality, assume $e_n\in\bdy\OM$ for
    all $n$ and $e_n\to e$, as $n\to\infty$. Fix $n\in\nat$, since $z^n_\nu\to e_n$
    as $\nu\to\infty$, we can find $N(n)\in\nat$ such that $d_h(z^n_\nu, e_n)<1/n$ for all
    $\nu\geq N(n)$. Also, since $i_{\koba}(z^n_\nu)\to\xi_n$, as $\nu\to\infty$, we can find $M(n)\geq N(n)$
    such that $d_{\koba}(\xi_n, i_{\koba}(z^n_\nu))<1/n$ for all $\nu\geq M(n)$. Here, $d_{\koba}$ is any distance that 
    metrizes the space $\clos{\OM}^{\koba}$. Consider the sequence given by $x_n=z^n_{M(n)}$ and observe
    since 
    \begin{align*}
        d_h(x_n, e)&\leq d_h(x_n, e_n)+d_h(e_n, e)<1/n+d_h(e_n, e) \ \ \ \text{and}\\
        d_{\koba}(i_{\koba}(x_n), \xi)&\leq d_{\koba}(i_{\koba}(x_n), \xi_n)+d_{\koba}(\xi_n, \xi)<1/n+d_{\koba}(\xi_n, \xi).
    \end{align*}    
    It follows that $x_n\to e$ and $i_{\koba}(x_n)\to\xi$ as $n\to\infty$, whence $F(\xi)=e$.  
    \smallskip

     \noindent{\bf Subcase (b)} $e$ is an end, and $e_n\in\bdy\OM$ for all $n\in\nat$.
     \smallskip

     \noindent We choose and fix an exhaustion $(K_j)_{j\geq 1}$ of $\clos{\OM}$ by 
     compact sets.
     Let $e=(F_j)_{j\geq 1}$ where, for each $j$, $F_j$ is a connected
     component of $\clos{\OM}\setminus K_j$ such that $F_{j+1}\subset F_j$. Recall the neighborhood basis
     for $e$ is given by the family $(\widehat{F}_j)_{j\geq 1}$ of subsets of $\clos{\OM}^{\rm End}$ where
     \begin{equation}\label{E:nbhdbasisend}
     \widehat{F}_j:=F_j\cup\{f\,|\,\text{$f=(G_\nu)_{\nu\geq 1}$ is an end
     of $\clos{\OM}$ such that $G_\nu=F_\nu \ \ \forall\nu=1,\ldots,j$}\}.
     \end{equation}
     Fix $j$, since $\widehat{F}_j$ is a neighborhood of $e$, there exists $N(j)\in\nat$ such that 
     $e_n\in\widehat{F}_j$ for all $n\geq N(j)$. Since $e_n\in\bdy\OM$, we have $e_n\in F_j$ for 
     all $n\geq N(j)$.
     By Lemma~\ref{lm:nbhdunbseq}, for each $j$, $F_j$ is a neighborhood of $e_n$ for all $n\geq N(j)$.
     Therefore, for each $j$, we can find large enough $M(j)\in\nat$ such that
     \[
     x_j:=z^{j+N(j)}_{M(j)}\in F_j \ \ \ 
     \text{and \ \ $d_{\koba} \Big(\xi_{j+N(j)}, i_{\koba}\big(z^{j+N(j)}_{M(j)}\big)\Big)<\frac{1}{j}$}. 
     \]
     Clearly $x_j\to e$ and $i_{\koba}(x_j)\to\xi$ as $j\to\infty$. Therefore, in this case too we have 
     $F(\xi)=e$. 
     \smallskip

     \noindent{\bf Subcase (c).} $e$ and $e_n$'s, $n\in\nat$, are ends.
     \smallskip

     \noindent Fix $n\in\nat$, since $i_{\koba}(z^n_\nu)\to\xi_n$ as $\nu\to\infty$,
     we can find $N_1(n)\in\nat$ such that
     \[
     d_{\koba}\big(i_{\koba}(z^n_\nu), \xi_n\big)<\frac{1}{n} \ \ \ \forall\nu\geq N_1(n).
     \]
     Moreover, we can also assume that $N_1(n)$ is strictly increasing as a function
     of $n$. Since $e_n$'s are end,
     we write $e_n=(E^n_\nu)_{\nu\geq 1}$. For each $n$, since $\widehat{E}^n_n$
     is a neighborhood of $e_n$ and $z^n_\nu\to e_n$, we can find $N_2(n)$,
     again strictly incresing in $n$, 
     such that $z^n_\nu\in\widehat{E}^n_n$ for all 
     $\nu\geq N_2(n)$. Let $N(n):=\max\{N_1(n), N_2(n)\}$ and consider the sequence 
     $x_n=z^n_{N(n)}$. Clearly $i_{\koba}(x_n)\to \xi$ and $x_n\in E^n_{n}$ for all $n$.
     We now show that $x_n\to e$ as $n\to\infty$. Recall the neighborhood basis
     for $e$ is given by the family $(\widehat{F}_j)_{j\geq 1}$ of subsets of $\clos{\OM}^{\rm End}$ as in \eqref{E:nbhdbasisend}.
     \smallskip

     \noindent\textbf{Claim:} For each $j\geq 1$, there exists 
      $N_j\in\nat$ such that $x_n\in\widehat{F}_j$ for all $n\geq N_j$.
    \smallskip
    
     \noindent Fix $j\geq 1$, since $e_n\to e$ as $n\to\infty$,
     there exists a $N_j$ such that $e_n\in\widehat{F}_j$ for all
     $n\geq N_j\geq j$.
     Choose $n\geq N_j$ arbitrary then $e_n\in\widehat{F}_j$ implies 
     that the end $e_n=(E^n_\nu)_{\nu\geq 1}$ satisfies $E^n_\nu=F_\nu$ for $\nu=1,2,...,j$. 
     %For $\nu=j$, $E^n_j=F_j$ and we have $x_n\in\ E^n_n$. 
     For $n\geq N_j\geq j$, we have
     $x_n\in E^n_n\subset E^n_j=F_j\subset\widehat{F}_j$. 
     Hence the claim follows. 
     \smallskip
     
     This implies that $x_n\to e$ as $n\to\infty$, and therefore, $F(\xi)=e$ in 
     this case too. This completes the proof. 
\end{proof}
In complete hyperbolic domains, a necessary condition for the existence
of the map $F$ satisfying condition~\ref{E:extcanodup},
as in the above theorem, is given by the following proposition. 
\begin{proposition}\label{prop:necmettoend}
Let $\OM \subsetneq \mfd{X}$ be a complete Kobayashi hyperbolic domain 
in a complex manifold $\mfd{X}$. 
Suppose there exists a continuous surjective map 
\(F: \clos{\OM}^{\koba} \lrarw \clos{\OM}^{\mathrm{End}}\) such that 
\begin{equation*}
F \circ i_{\koba} = i_\OM \quad \text{on } \OM.
\end{equation*}
Then every geodesic ray \(\gamma: \mathbb{R}_+ \lrarw \OM\) lands in \(\clos{\OM}^{\mathrm{End}}\).
\end{proposition}

\begin{proof}
Given a geodesic ray \(\gamma\), consider the family of functions
\[
\left( \koba_\OM(\cdot, \gamma(t)) - 
\koba_\OM(\gamma(t), \gamma(0)) \right)_{t \geq 0}.
\]
Each function in the family is 1-Lipschitz and
the family is monotonically decreasing (in \(t\)) and
locally bounded above. Hence, by the Arzelà–Ascoli theorem,
the family converges uniformly on compact subsets to a 1-Lipschitz
function \(F_\gamma\), given by
\[
F_\gamma(\cdot) := \lim_{t \to \infty} \left( \koba_\OM(\cdot, \gamma(t)) - \koba_\OM(\gamma(t), \gamma(0)) \right).
\]
This function defines a horofunction, and let \(\xi := [F_\gamma]\) denote
its equivalence class. Note that in our notation,
\(h_{p,\,\xi}(\cdot) = F_\gamma(\cdot)\), where \(p = \gamma(0)\).
Define the cluster set of \(\gamma\) in \(\clos{\OM}^{\mathrm{End}}\) as
\[
C(\gamma) := \left\{ x \in \partial \clos{\OM}^{\mathrm{End}} \;\middle|\; \exists\, (t_n) \to \infty \text{ such that } \gamma(t_n) \to x \right\}.
\]
For any \(x \in C(\gamma)\), we must have \(\xi \in \mathscr{H}(x)\)
by the construction. Hence,
\[
\xi \in \bigcap_{x \in C(\gamma)} \mathscr{H}(x).
\]
By Theorem~\ref{thm:mettoend}, we know that
\(\mathscr{H}(x) \cap \mathscr{H}(y) = \emptyset\) whenever
\(x \ne y \in \clos{\OM}^{\mathrm{End}}\). 
It follows that \(C(\gamma)\) must be a singleton. 
Therefore, the geodesic ray \(\gamma\) lands at a unique point in
\(\partial \clos{\OM}^{\mathrm{End}}\).
\end{proof}
We remind the reader that the geodesic visibility property also ensures that
geodesic rays land at a point in the boundary, see \cite[Lemma~3.1]{BNT:2022}. 
However, our next result\,---\,which provides
a sufficient condition for the existence of the map \(F\) in 
Theorem~\ref{thm:mettoend}\,---\,shows that geodesic 
visibility alone suffices to guarantee such an extension.
\begin{proposition}\label{prop:suff}
Let $\OM\subsetneq\mfd{X}$ be a complete Kobayashi hyperbolic domain.
Suppose for a fixed $p\in\OM$, and any $x\in\bdy\clos{\OM}^{\rm End}$
we have 
\begin{equation}\label{E:intbighorosing}
\bigcap_{R>0}\clos{H^b_p(x, R)}=\{x\}. 
\end{equation}
Then, for any $\xi\in\mathscr{H}(x)$, we also have 
\begin{equation}\label{E:inthoroballcent}
\bigcap_{R>0}\clos{H_p(\xi, R)}=\{x\}.
\end{equation}
In particular, there exists a continuous
surjective map $F:\clos{\OM}^{\koba}\lrarw\clos{\OM}^{\rm End}$
such that 
\[
F\circ i_{\koba}={i}_\OM \ \ \ \text{on \ \ $\OM$}.
\]
\end{proposition}
\begin{proof}
Given $x\in\bdy\clos{\OM}^{\rm End}$, 
for any $R>0$, we have 
\[
 H^b_p(x, R)=\bigcup_{\xi\in\mathscr{H}(x)}H_p(\xi, R),
 \]
where, we recall:
$H_p(\xi, R):=\big\{z\in\OM : h_{p,\,\xi}(z)<\tfrac{1}{2}\log R\big\}$. 
Here, $h_{p,\,\xi}$ is the unique 
horofunction representing $\xi$ and vanishing at $p$.
Therefore, for any $R>0$, 
\begin{equation}\label{E:inclusionofhoroball}  \clos{H_p(\xi,R)}\subset\clos{H^b_p(x,R)}\ \ \ \text{$\implies$}\ \ \ \bigcap_{R>0}\clos{H_p(\xi,R)}\subset\bigcap_{R>0}\clos{H^b_p(x,R)}=\{x\}
\end{equation}
We now claim that $\bigcap_{R>0}\clos{H_p(\xi, R)}$ is non-empty. To see that,
for each $n$, choose $z_n\in H_p(\xi,1/n)$.
This is possible since for any $R>0$, each $H_p(\xi,R)$ is non-empty,
according to Lemma~\ref{lm:horoneg}. 
By passing to a subsequence, we can assume 
$z_n\to p\in\bdy\clos{\OM}^{\rm End}$. 
Note that for any $R>0$, there exists $N(R)$
such that $z_n\in \clos{H_p(\xi,R)}$ for all $n\geq N(R)$. 
Therefore, $p\in\clos{H_p(\xi,R)}$ for all $R>0$ whence the claim.
\smallskip

Therefore, the inclusion \eqref{E:inclusionofhoroball} and
the above claim implies \eqref{E:inthoroballcent}.
Let $x\neq y\in\bdy\clos{\OM}^{\rm End}$ and let
$\xi\in\mathscr{H}(x)\cap\mathscr{H}(y)$. Then, it follows, from 
\eqref{E:inthoroballcent} that
\[
\bigcap_{R>0}\clos{H_p(\xi, R)}=\{x\}=\{y\},
\]
which is a contradiction. Hence, 
$\mathscr{H}(x)\cap\mathscr{H}(y)=\emptyset$ whenever
$x\neq y\in\bdy\clos{\OM}^{\rm End}$. 
The result now follows from an application of Theorem~\ref{thm:mettoend}.
\end{proof}
As noted earlier, both geodesic
visibility domains and strictly convex domains satisfy condition
\eqref{E:intbighorosing}. Therefore, it follows from the above proposition that 
the canonical map extends as
a continuous map from $\clos{\OM}^{\koba}$ to $\clos{\OM}^{\rm End}$. 
Note that, it also implies that, geodesic rays land in a strictly convex domains. 
\smallskip

We also record the following result, a part of
which\,---\,namely, the equivalence of (1) and (3)\,---\,is known in the setting of bounded domains in 
$\C^n$. For completeness, we briefly sketch a proof in our general setting.
\begin{theorem}\label{T:endtomet}
Let $\OM\subsetneq\mfd{X}$ be a Kobayashi hyperbolic domain of a complex manifold 
$\mfd{X}$ and let $o\in\OM$ be a fixed point. Then the following are equivalent:
\begin{enumerate}
    \item For every $x\in\bdy\clos{\OM}^{\rm End}$, the limit 
    \[
    \lim_{w\to x}(\koba_\OM(z, w)-\koba_\OM(o, w)) \ \ \ \text{exists for all $z\in\OM$}.
    \]
    \item For every $x\in\bdy\clos{\OM}^{\rm End}$, the set $\mathscr{H}(x)$ is a singleton set. 
    \smallskip

    \item The map $i_{\koba}$ extends continuously from $\clos{\OM}^{\rm End}\lrarw\clos{\OM}^{\koba}$. 
\end{enumerate}
\end{theorem}
\begin{proof}
Suppose $(1)$ holds. Fix $x\in\bdy\clos{\OM}^{\rm End}$ and let $\xi_1, \xi_2\in\mathscr{H}(x)$. Then 
there exist sequences $(x_n), (y_n)$, both converging to $x$, such that $[h_{o,\,\xi_1}]=\xi_1$ 
and $[h_{o,\,\xi_2}]=\xi_2$, where 
\begin{align*}
h_{o,\,\xi_1}(\cdot)&=\lim_{n\to\infty}\big(\koba_\OM(\cdot, x_n)-\koba_\OM(o, x_n)\big),\\
h_{o,\,\xi_2}(\cdot)&=\lim_{n\to\infty}\big(\koba_\OM(\cdot, y_n)-\koba_\OM(o, y_n)\big)
\end{align*}
are the horofunctions representing $\xi_1, \xi_2$, and which vanish at $o$. Clearly, by our assumption
$h_{o,\,\xi_1}(\cdot)=h_{o,\,\xi_2}(\cdot)$ and hence $\xi_1=\xi_2$. This implies $\mathscr{H}(x)$ is
singleton and since $x\in\bdy\clos{\OM}^{\rm End}$ is arbitrary, we have $(2)$. Now suppose $(2)$ holds
for each $x\in\bdy\clos{\OM}^{\rm End}$. Fix an $x\in\bdy\clos{\OM}^{\rm End}$ and suppose there exists $z\in\OM$
such that 
\begin{equation}\label{E:limnotexist}
 \liminf_{w\to x}(\koba_\OM(z, w)-\koba_\OM(o, w))\neq \limsup_{w\to x}(\koba_\OM(z, w)-\koba_\OM(o, w))
 \end{equation}
Let $(x_n), (y_n)$ be sequences, both converging to $x$, such that 
\begin{align*}
\liminf_{w\to x}(\koba_\OM(z, w)-\koba_\OM(o, w))&=
\lim_{n\to\infty}\big(\koba_\OM(z, x_n)-\koba_\OM(o, x_n)\big)\\
\limsup_{w\to x}(\koba_\OM(z, w)-\koba_\OM(o, w))&=
\lim_{n\to\infty}(\koba_\OM(z, y_n)-\koba_\OM(o, y_n)).
\end{align*}
By passing to subsequences, we can find $\xi_1, \xi_2\in\mathscr{H}(x)$ such that 
\begin{align*}
h_{o,\,\xi_1}(\cdot)&=\lim_{n\to\infty}\big(\koba_\OM(\cdot, x_n)-\koba_\OM(o, x_n)\big),\\
h_{o,\,\xi_2}(\cdot)&=\lim_{n\to\infty}\big(\koba_\OM(\cdot, y_n)-\koba_\OM(o, y_n)\big).
\end{align*}
Note that $\xi_1=\xi_2$ if and only if $h_{o,\,\xi_1}(\cdot)=h_{o,\,\xi_2}(\cdot)$ which is not the
case due to \eqref{E:limnotexist}. Therefore, $\xi_1\neq\xi_2\in\mathscr{H}(x)$, a contradiction to 
our assumtion. Hence $(2)$ implies $(1)$.
\smallskip

Observe that $(3)$ clearly implies $(2)$. Now suppose $(2)$ holds and consider the map
$G:\clos{\OM}^{\rm End}\lrarw\clos{\OM}^{\koba}$ defined by
    \[
    G(x):=
    \begin{cases}
        i_{\koba}(x), & \ \ \ x\in\OM,\\ 
        \xi, & \ \ \ \text{otherwise, where $\mathscr{H}(x)=\{\xi\}$.}
    \end{cases}
    \]
It is a routine\,---\,following the same strategy as in the proof
of Theorem~\ref{thm:mettoend}\,---\,to check that $G$ is a continuous map. 
\end{proof}

\end{document}